\documentclass[leqno]{siamltex}
\usepackage{amsmath}
\usepackage{graphicx}
\usepackage{mathrsfs}
\usepackage{bm}
\usepackage{float}
\usepackage{amsfonts,amssymb}
\usepackage{dsfont}
\usepackage{pifont}
\usepackage{tikz}
\usepackage{wrapfig} 
\usepackage{hyperref}
\usepackage{multirow}
\usepackage{lineno}
\usepackage{mathtools}
\usepackage{appendix}
\usepackage{color}
\usepackage{caption}
\usepackage{subfigure}

\newtheorem{remark}{Remark}
\newtheorem{example}{Example}

\newcommand{\vertiii}[1]{{\left\vert\kern-0.25ex\left\vert\kern-0.25ex\left\vert #1
    \right\vert\kern-0.25ex\right\vert\kern-0.25ex\right\vert}}

\title{Finite element method with the total stress variable for Biot's consolidation model}

\author{Wenya Qi\thanks{School of Mathematics and Statistics, Lanzhou University, Lanzhou 730000, China(qiwy16@lzu.edu.cn). The research of W.Q. was partially supported by China Scholarship Council, Grant Number: 201906180039 and  National Natural Science Foundation of China (Grant No.11471150).}
 \and Padmanabhan Seshaiyer
\thanks{Department of Mathematical Sciences, George Mason University, Fairfax, VA 22030, USA (pseshaiy@gmu.edu).}
 \and Junping Wang
\thanks{Division of Mathematical Sciences, National Science Foundation, Alexandria, VA 22314, USA(jwang@nsf.gov), and Department of Mathematical Sciences, George Mason University, Fairfax, VA 22030, USA. The research of J.W. was supported by the NSF IR/D program, while working at National Science Foundation. However, any opinion, finding, and conclusions or recommendations expressed in this material are those of the author and do not necessarily reflect the views of the National Science Foundation.}}

\begin{document}

\maketitle
\begin{abstract}
 In this work, semi-discrete and fully-discrete error estimates are derived for the Biot's consolidation model described using a three-field finite element formulation. The fields include displacements, total stress and pressure. The model is implemented using a backward Euler discretization in time for the fully-discrete scheme and validated for benchmark examples. Computational experiments presented verifies the convergence orders for the lowest order finite elements with discontinuous 
 and continuous finite element appropriation for the total stress.  
 \end{abstract}



\pagestyle{myheadings}

\section{Introduction}\label{sec:1}
Biot's fundamental equations for the soil consolidation process that describes the gradual adaptation of the soil to a load variation have been established in \cite{Biot1941}. The mechanism of consolidation phenomenon described using a linear isotropic model is identical with the process of squeezing water out of an elastic porous medium in many cases. This solid-fluid coupling was extended to general anisotropy case in \cite{Biot1955}. Such poroelasticity models have a lot applications in many areas including geomechanics \cite{Soltan2007}, medicine \cite{Nagashima1990}, biomechanics \cite{Swan2003}, reservoir engineering \cite{Lewis1993}. The Biot's consolidation models have also been used to combine transvascular and interstitial fluid movement with the mechanics of soft tissue in \cite{Netti1997}, which can be applied to improve drug delivery in solid tumors.
Existence, uniqueness, and regularity theory were developed in \cite{Showalter2000} for poroelasticity and quasi-static problem in thermoelasticity.
In  \cite{Leiderman2006}, experiments with finite element method for the model in \cite{Netti1997} have demonstrated the effects of fluid flow on the spatio-temporal patterns of soft-tissue elastic strain under a variety of physiological condition, which emphasized simulations relevant to a quasistatic elasticity imaging for the characterization of fluid flow in solid tumors. 

Biot's consolidation model has been considered by many researchers using finite element methods. In \cite{Yokoo1971}, a variational principle and the finite element method for a model with applications to a nonhomogeneous, anisotropic soil were developed. 
The fully discretization with backward Euler time discrete finite element method has been carried out and the existence and uniqueness were proved in \cite{zenisek1984}.
Moreover, the simplest Taylor-Hood finite elements were employed.
The stability and error estimates for the semi-discrete and fully-discrete finite element approximations were derived in \cite{Murad1992} and \cite{Murad1994}.
Decay functions and post processing technique also were employed to improve the pore pressure accuracy.
 With the displacement and pore pressure fields as unknowns, the short and long time behavior of spatially discrete finite element solutions have been studied in \cite{murad1996}. Asymptotic error estimates have been derived for both stable and unstable combinations of the finite element spaces.  
 For the coupled problem, a least squares mixed finite element method was presented and analyzed in \cite{Korsawe2005} with the unknowns fluid displacement, stress tensor,  flux, and  pressure. 
 In \cite{Phillips2007}, coupling of mixed and continuous Galerkin finite element methods  for pressure and displacements have been formulated deriving the convergence error estimates in time continuous setting.  Methods for coupling mixed and discontinuous Galerkin have been presented in \cite{Phillips2008}.
The error estimates for a fully-discrete stabilized discontinuous Galerkin method were obtained with the unknowns pressure and displacement in \cite{CHEN2013}.
In \cite{Wheeler2014}, a discretization method in irregular domains with general grids for discontinuous full tensor permeabilities was developed.
A new mixed finite element method for Biot's consolidation problem in four variables was proposed in \cite{sonyi2014} and later, a three field mixed  finite element which was free of pressure oscillations and Poisson locking has been proposed \cite{sonyi2017}. 
The priori error estimates that were robust for material parameters were provided in \cite{LEE2016} with a four-field mixed method formulation.
A three field finite element formulation with nonconforming finite element space for the displacements was considered in \cite{Hu2017}.
Based on the parameter dependent norms,  the parameter-robust stability  was established in \cite{Hong2018}, and parameter robust inf-sup stability and strong mass conservation were derived for three field mixed discontinuous Galerkin discretizations. 
A stabilized finite element method with equal order elements for the unknowns pressure and displacement was proposed in \cite{Chen2018} to reduce the effects of non-physical oscillations.
Combining the mixed method with symmetric interior penalty discontinuous Galerkin method obtained a H(div) conforming finite element method in \cite{Kanschat2018}. The method achieved strong mass conservation.

Moreover, in \cite{Oyarzua2016}, the total stress (or the soil pressure) expressed as a combination of the divergence of the velocity and pressure has been introduced coupling the solid and fluid robustly for Biot’s consolidation problem with the unknown displacement, pressure, and volumetric stress. Using a Fredholm argument for a static model, error estimates were derived independently of the Lam$\acute{e}$ constants for both continuous and discrete formulations.
A three field formulation of Biot's model with the total stress variable has also been proposed in \cite{Lee2017},  which developed a parameter-robust block diagonal preconditioner for the associated discrete systems. Then, in \cite{Lee2019}, 
a priori error estimates for semi-discrete scheme has been presented by introducing the total pressure variable for quasi-static multiple-network poroelasticity equations.
Our goal in this paper is to emphasize the time dependence of field variables in error estimates, so we choose to consider the fully-discrete scheme. Thus, using the total stress as a new variable for the three field formulation, we give the error estimates for semi-discrete and fully-discrete with backward Euler time discretization schemes. 

Let $\Omega$ be an open bounded polygonal or polyhedral domain in $\mathbb{R}^d,~d=2,3$. Biot's consolidation model is described  as follows: Find the displacement vector $\mathbf{u}$ and the fluid pressure $p$ in
\begin{align}
-\nabla \cdot(2\mu\varepsilon(\mathbf{u})+\lambda\nabla\cdot \mathbf{u} \mathbf{I}- p\mathbf{I})&=\mathbf{f},~~~~~~~~~~\mbox{in}~\Omega\times(0,\bar{T}], \label{eq:problem1.1}\\
\nabla\cdot(D_t\mathbf{u})-\nabla\cdot(k\nabla p)&=g,~~~~~~~~~~\mbox{in}~\Omega\times(0,\bar{T}], \label{eq:problem1.2}\\
\mathbf{u}=\mathbf{0}, ~~p&=0, ~~~~~~~~~~\mbox{on}~\Gamma_D\times(0,\bar{T}],\label{eq:boundarycondition1.3}\\
(2\mu\varepsilon(\mathbf{u})+\lambda\nabla\cdot \mathbf{u} \mathbf{I}- p\mathbf{I})\cdot \mathbf{n}&=\mathbf{\beta},~~~~~~~~~~\mbox{on}~\Gamma_N\times(0,\bar{T}],\label{eq:boundarycondition1.4}\\
(\kappa\nabla p)\cdot \mathbf{n}&=\mathbf{\gamma},~~~~~~~~~~\mbox{on}~\Gamma_N\times(0,\bar{T}],\label{eq:boundarycondition1.5}
\end{align}
where  $\Gamma_D$ and $\Gamma_N$ are disjoint closed subsets with $\Gamma_D\cup \Gamma_N=\partial \Omega$ and the Dirichlet  boundary $\vert \Gamma_D\vert\neq0$. Here, $\varepsilon(\mathbf{u})=\frac{1}{2}(\nabla\mathbf{u}+(\nabla\mathbf{u})^T)$ is the strain tensor expressed in terms of symmetrized gradient of displacements, $\kappa\in(0, \infty)$ is the permeability of the porous solid and parameters $\mu\in(0, \infty)$, $\lambda\in(0, \infty)$ are the elastic Lam$\acute{e}$ constants. The right hand side term $\mathbf{f}$ in (\ref{eq:problem1.1}) represents the density of the applied body forces, and  the source term $g$ in (\ref{eq:problem1.2}) represents a forced fluid extraction or injection. The outward unit normal vector on $\partial \Omega$ is denoted by $\mathbf{n}$. 

Next, we introduce the coupling between the solid and fluid using $q=-\lambda\nabla\cdot\mathbf{u}+p$, where $q$ is the total stress, $\mathbf{u}$ is the displacement and $p$ is the fluid  pressure \cite{Oyarzua2016, Lee2017}. This can be rewritten as 
\begin{equation}\label{eq:coupling}
\lambda^{-1}(q-p)+\nabla\cdot\mathbf{u}=0.
\end{equation}
For $x\in\Omega$, the initial conditions  are given by
\begin{equation}\label{initialcondition}
\begin{aligned}
\mathbf{u}(x,0)&=\varphi,\\
p(x,0)&=\phi.
\end{aligned}
\end{equation}
Hence, the initial condition for the total stress is $q(x, 0)=-\lambda\nabla\cdot\mathbf{\varphi}+\phi$.  Denote the spaces $[H^1_{0,D}(\Omega)]^d=\{\mathbf{v}\in [H^1(\Omega)]^d, ~\mathbf{v}=\mathbf{0} ~\mbox{on} ~\Gamma_D\}$
and $H^1_{0,D}(\Omega)=\{ v\in H^1(\Omega), ~v=0~\mbox{on} ~\Gamma_D\}$. 

Multiplying \eqref{eq:problem1.1}, \eqref{eq:coupling} and \eqref{eq:problem1.2} by test functions, integrating by parts and applying boundary conditions yield the following weak formulation: Find $\mathbf{u}\in[H^1_{0,D}(\Omega)]^d$, $q\in L^2(\Omega)$,  $p\in H^1_{0,D}(\Omega)$ such that
\begin{equation}\label{weakform_FEM1}
\begin{aligned}
2\mu(\varepsilon(\mathbf{u}), \varepsilon(\mathbf{v}))-(q,\nabla\cdot\mathbf{v})&=(\mathbf{f},\mathbf{v})+\langle \beta, \mathbf{v}\rangle_{\Gamma_N}, ~~~~~~\forall \mathbf{v}\in[H^1_{0,D}(\Omega)]^d,\\
(\lambda^{-1}(q-p),w_q)+(\nabla\cdot\mathbf{u},w_q)&=0,~~~~~~~~~~~~~~~~~~~~~~~~~\forall w_q\in L^2(\Omega),\\
-(\lambda^{-1}(q_t-p_t),w_p)+\kappa(\nabla p,\nabla w_p)&=(g, w_p)+\langle \gamma, w_p\rangle_{\Gamma_N},~~\forall w_p\in H^1_{0,D}(\Omega).\\
\end{aligned}
\end{equation}
In Section \ref{sec:2}, we present the semi-discrete and fully-discrete finite element formulation for system (\ref{weakform_FEM1}) with the unknown displacement, total stress and pressure and prove uniqueness of the solution for each of these schemes. Section \ref{sec:3} presents the derivation and analysis of the error estimates for both the semi-discrete and fully-discrete schemes. In Section \ref{sec:4} we present computational experiments on benchmark problems that validate the theoretical convergence rates with respect to mesh size $h$ and time step $\tau$.

In the paper, we denote the arbitrary constants by $\epsilon_i\geq0$,  where $i$ is positive integer and $C$ is a constant which is independent of time step $\tau$ and mesh size $h$. Let ${P}_{k}$ be the space of polynomials of degree less than or equal to $k$ in all variables. Moreover, let $\Vert\cdot\Vert$ be the norm in $L^2(\Omega)$
space and $\Vert\cdot\Vert_k$ be the norm in $H^k(\Omega)$ space. Denote the space-time space by $L^{k}(0,\bar{T}; V)$ for a Banach space $V$ (see details in \cite{Quarteroni2008}). 

\section{Finite Element Discretization and Uniqueness}\label{sec:2}
Let $\mathcal{T}_h$ be a regular and quasi-uniform triangulation of domain $\Omega$ into triangular or tetrahedron elements \cite{Brenner2008, Ciarlet1978}. For each element $T\in\mathcal{T}_h$, $h_T$ is its diameter and $h=\max_{T\in\mathcal{T}_h}h_T$ is the mesh size of triangulation $\mathcal{T}_h$.  We consider the following finite element spaces on $\mathcal{T}_h$,
 \begin{equation*}
\begin{aligned}
 \mathbf{U}_{h}:=&\{\mathbf{v}\in[H^1_{0,D}(\Omega)]^d\cap [C^0(\Omega)]^d: \mathbf{v}\mid_{T}\in[{P}_{k}(T)]^d, \forall T \in \mathcal{T}_{h}\},\\
 Z_{h}:=&\{q\in L^2(\Omega): q\mid_{T}\in{P}_{l}(T), \forall T \in \mathcal{T}_{h}\},\\
P_{h}:=&\{p\in H^1_{0,D}(\Omega)\cap C^0(\Omega): p\mid_{T}\in{P}_{k-1}(T), \forall T \in \mathcal{T}_{h}\},\\
\end{aligned}
\end{equation*}
where $k\geq2,~l\geq 0.$
In order to describe the initial conditions of discretization schemes, we define the following two projection operators.  Let us define the Stokes projection $Q_h^{\mathbf{u}}: [H^1_{0,D}(\Omega)]^d \rightarrow \mathbf{U}_h$ and 
$Q_h^q: L^{2}(\Omega) \rightarrow Z_h$ by
 \begin{equation}\label{stoke_proj}
\begin{aligned}
2\mu(\varepsilon(Q_h^{\mathbf{u}}\mathbf{u}), \varepsilon(\mathbf{v}))-(Q_h^q q,\nabla\cdot\mathbf{v})&=2\mu(\varepsilon(\mathbf{u}), \varepsilon(\mathbf{v}))-(q,\nabla\cdot\mathbf{v}), ~~~~~~\forall \mathbf{v}\in\mathbf{U}_h,\\
(w_q, \nabla\cdot Q_h^{\mathbf{u}}\mathbf{u})&=(w_q, \nabla\cdot\mathbf{u}),~~~~~~~~~~~~~~~~~~~~~~~~~~\forall w_q\in Z_h.
\end{aligned}
\end{equation}

Also the elliptic projection 
$Q_h^{p}: H^1_{0,D}(\Omega) \rightarrow P_{h}$ is defined with the following properties,
\begin{equation}\label{elliptic_proj}
\begin{aligned}
(\nabla Q_h^{p}{p},\nabla{w_p})&= (\nabla {p},\nabla{w_p}), ~~\forall w_p\in {P}_{h}.
 \end{aligned}
\end{equation}
Hence, given a suitable approximation of initial conditions $\mathbf{u}_h(0)=Q_h^{\mathbf{u}}\varphi$, $q_h(0)=Q_h^q(-\lambda\nabla\cdot\varphi+\phi)$ and  $p_h(0)=Q_h^p\phi$, the semi-discrete scheme corresponding to a three field formulation \eqref{weakform_FEM1} is for all $t\in[0, \bar{T}]$,  to seek $\mathbf{u}_h(t)\in \mathbf{U}_{h}$, $q_h(t)\in Z_{h}$, $p_h(t)\in P_{h}$ such that
\begin{equation}\label{FEM1_semi}
\begin{aligned}
2\mu(\varepsilon(\mathbf{u}_h), \varepsilon(\mathbf{v}))-(q_h,\nabla\cdot\mathbf{v})&=(\mathbf{f},\mathbf{v})+\langle \beta, \mathbf{v}\rangle_{\Gamma_N}, ~~~~~\forall \mathbf{v}\in\mathbf{U}_{h},\\
(\lambda^{-1}(q_{h}-p_{h}),w_q)+(\nabla\cdot\mathbf{u}_{h},w_q)&=0,~~~~~~~~~~~~~~~~~~~~~~~~~\forall w_q\in Z_{h},\\
-(\lambda^{-1}(q_{ht}-p_{ht}),w_p)+\kappa(\nabla p_h,\nabla w_p)&=(g, w_p)+\langle \gamma, w_p\rangle_{\Gamma_N},~~\forall w_p\in P_{h}.\\
\end{aligned}
\end{equation}
To obtain a fully-discrete formulation, we denote time step by $\tau$, and $t^n= n \tau$, where $n$ is non-negative integer. Thus, given a  suitable approximation of initial conditions $\mathbf{u}
^0=Q_h^{\mathbf{u}}\varphi$,  $q^0=Q_h^q(-\lambda\nabla\cdot\varphi+\phi)$ and $p^0=Q_h^p\phi$,  the fully-discrete scheme with backward Euler time discretization corresponding to the three field formulation \eqref{weakform_FEM1} is to find  $\mathbf{u}^n\in \mathbf{U}_{h}$, $q^n\in Z_{h}$, $p^n\in P_{h}$ such that
\begin{equation}\label{FEM1_fully}
\begin{aligned}
2\mu(\varepsilon(\mathbf{u}^n), \varepsilon(\mathbf{v}))-(q^n,\nabla\cdot\mathbf{v})&=(\mathbf{f}^n,\mathbf{v})+\langle \beta^n, \mathbf{v}\rangle_{\Gamma_N}, ~~~~~~~\forall \mathbf{v}\in\mathbf{U}_{h},\\
(\lambda^{-1}( q^n-p^n),w_q)+(\nabla\cdot {\mathbf{u}^n},w_q)&=0,~~~~~~~~~~~~~~~~~~~~~~~~~~~~~\forall w_q\in Z_{h},\\
-(\lambda^{-1}(\bar{\partial}_t q^n-\bar{\partial}_t p^n),w_p)+\kappa(\nabla p^n,\nabla w_p)&=(g^n, w_p)+\langle \gamma^n, w_p\rangle_{\Gamma_N},~~~\forall w_p\in P_{h},\\
\end{aligned}
\end{equation}
where $\bar{\partial}_t{\mathbf{u}^n}: =\displaystyle\frac{\mathbf{u}^n-\mathbf{u}^{n-1}}{\tau}$ and $f^n: =f(x, t^n),~x\in\Omega$.

Then, we present an inequality, which will be useful in the uniqueness and convergence analysis.
\begin{lemma}\label{korn_leml}(Korn's inequality) \cite{Nitsche1981, Brenner2004} For each $\mathbf{u}\in[H^1(\Omega)]^d$, there exists a positive constant $C$ such that
\begin{equation}\label{korn_inequa}
\vert \mathbf{u}\vert_1\leq C(\Vert\varepsilon(\mathbf{u})\Vert+\Vert \mathbf{u} \Vert).
\end{equation}
\end{lemma}

Next we use Lemma \ref{korn_leml} to prove the uniqueness for semi-discrete (\ref{FEM1_semi}) and fully-discrete (\ref{FEM1_fully}) schemes.
\begin{theorem}
For each $t\in (0,\bar{T} ]$, the semi-discrete scheme \eqref{FEM1_semi} has a unique solution.
\end{theorem}
\begin{proof}
We need to prove that the homogeneous problem of \eqref{FEM1_semi} has only the trivial solution. Taking the time derivative of the second equation in \eqref{FEM1_semi},
the homogeneous problem is rewritten as seeking $\mathbf{u}_h\in \mathbf{U}_{h}$, $q_h\in Z_{h}$, $p_h\in P_{h}$ with $\mathbf{u}_h(0)=0$, $q_h(0)=0$, $p_h(0)=0$ such that
\begin{equation}\label{FEM1_semi_homo}
\begin{aligned}
2\mu(\varepsilon(\mathbf{u}_h), \varepsilon(\mathbf{v}))-(q_h,\nabla\cdot\mathbf{v})&=0, ~~~~~\forall \mathbf{v}\in\mathbf{U}_{h},\\
(\lambda^{-1}(q_{ht}-p_{ht}),w_q)+(\nabla\cdot\mathbf{u}_{ht},w_q)&=0,~~~~\forall w_q\in Z_{h},\\
-(\lambda^{-1}(q_{ht}-p_{ht}),w_p)+\kappa(\nabla p_h,\nabla w_p)&=0,~~~~\forall w_p\in P_{h}.\\
\end{aligned}
\end{equation}
Using the test functions $\mathbf{v}=\mathbf{u}_{ht}$, $w_q=q_h$ and $w_p=p_h$ in \eqref{FEM1_semi_homo} and simplifying, we can derive the following identity 
$$\mu\frac{d}{dt}\Vert\varepsilon(\mathbf{u}_h)\Vert^2+\frac{\lambda^{-1}}{2}\frac{d}{dt}\Vert q_h-p_h\Vert^2+\kappa\Vert\nabla p_h\Vert^2=0.$$
Integrating the above identity over $(0, t)$, one finds
$$\mu\Vert\varepsilon(\mathbf{u}_h(t))\Vert^2+\frac{\lambda^{-1}}{2}\Vert q_h(t)-p_h(t)\Vert^2+\kappa\int_0^t\Vert\nabla p_h\Vert^2ds=0.$$
Therefore, with the conditions $\mu\in(0, \infty)$,   $\lambda\in(0, \infty)$, $\kappa\in(0, \infty)$, we have
$$\Vert\varepsilon(\mathbf{u}_h(t))\Vert=0,~ \Vert q_h(t)-p_h(t)\Vert=0~ \mbox{and}~\Vert\nabla p_h\Vert=0.$$ Then,
using Korn's inequality from Lemma \ref{korn_leml} when $\vert\Gamma_D\vert\neq 0$ leads to $\mathbf{u}_h(t)=0, q_h(t)=0, p_h(t)=0$.
\end{proof}
\begin{theorem}
For $t^N\in(0,\bar{T}]$,  the fully-discrete scheme \eqref{FEM1_fully} has  a unique solution.
\end{theorem}
\begin{proof}
Similar to the semi-discrete, with
$\mathbf{u}^0=0$, $q^0=0$ and $p^0=0$, we rewrite the second equation of \eqref{FEM1_fully} and consider the homogeneous problem for fully-discrete scheme \eqref{FEM1_fully} is  to find $\mathbf{u}^n\in \mathbf{U}_{h}$, $q^n\in Z_{h}$, $p^n\in P_{h}$ such that
\begin{equation}\label{FEM1_fully_homo}
\begin{aligned}
2\mu(\varepsilon(\mathbf{u}^n), \varepsilon(\mathbf{v}))-(q^n,\nabla\cdot\mathbf{v})&=0, ~~~~~~~\forall \mathbf{v}\in\mathbf{U}_{h},\\
(\lambda^{-1}(\bar{\partial}_t q^n-\bar{\partial}_t p^n),w_q)+(\nabla\cdot\bar{\partial}_t{\mathbf{u}^n},w_q)&=0,~~~~~~\forall w_q\in Z_{h},\\
-(\lambda^{-1}(\bar{\partial}_t q^n-\bar{\partial}_t p^n),w_p)+\kappa(\nabla p^n,\nabla w_p)&=0,~~~~~~\forall w_p\in P_{h}.\\
\end{aligned}
\end{equation}
Choosing $\mathbf{v}=\tau\bar{\partial}_t\mathbf{u}^n$, $w_q=\tau q^n$ and $w_p=\tau p^n$, equation (\ref{FEM1_fully_homo}) can be simplified to 
$$\mu\Vert\varepsilon(\mathbf{u}^n)\Vert^2-\mu\Vert\varepsilon(\mathbf{u}^{n-1})\Vert^2+\frac{\lambda^{-1}}{2}\Vert q^n-p^n\Vert^2
-\frac{\lambda^{-1}}{2}\Vert q^{n-1}-p^{n-1}\Vert^2+\kappa\tau\Vert\nabla p^n\Vert^2\leq 0.$$
Here, we have used the inequality 
$$(\varepsilon(\mathbf{u}^n), \varepsilon(\mathbf{u}^n)-\varepsilon(\mathbf{u}^{n-1}))
\geq\frac{1}{2}(\Vert\varepsilon(\mathbf{u}^n)\Vert^2-\Vert\varepsilon(\mathbf{u}^{n-1})\Vert^2).$$
Summing over $n$ from $1$ to $N$, it follows that
\begin{equation*}
\begin{aligned}
\mu\Vert\varepsilon(\mathbf{u}^N)\Vert^2+\frac{\lambda^{-1}}{2}\Vert q^N-p^N\Vert^2+
\kappa\tau\sum_{n=1}^N\Vert\nabla p^n\Vert^2\leq 0.
\end{aligned}
\end{equation*}
Note that the assumptions $\mu\in(0, \infty)$,   $\lambda\in(0, \infty)$, $\kappa\in(0, \infty)$, $\vert\Gamma_D\vert\neq 0$ and using Korn's inequality \eqref{korn_inequa} from Lemma \ref{korn_leml}, 
we have $\mathbf{u}^N=0$, $p^N=0$ and $q^N=0$.
\end{proof}
\section{Error Estimates}\label{sec:3}
In order to derive the error estimates, we first show the inf-sup condition for space pair $( [H^1_{0,D}(\Omega)]^d, L^2(\Omega))$  with $\vert\Gamma_N\vert>0$. Denote $b(\mathbf{v}, w_q):=(w_q, \nabla\cdot\mathbf{v})$.  For each $w_q\in L^2(\Omega)$, there exists $\mathbf{v}\in [H^1_{0,D}(\Omega)]^d$ satisfying $w_q=\nabla\cdot\mathbf{v}$ and $\vert\mathbf{v}\vert_1\leq C\Vert w_q\Vert$. Thus, we can deduce that for $\forall w_q\in L^2(\Omega)$\cite{Girault1979, Oyarzua2016}
\begin{equation*}
\begin{aligned}
\sup_{0\neq\mathbf{v}\in [H^1_{0,D}(\Omega)]^d}\frac{b(\mathbf{v}, w_q)}{\vert\mathbf{v}\vert_1}\geq C\Vert w_q\Vert.
\end{aligned}
\end{equation*}
Since  not all the discrete finite element spaces meet the inf-sup condition,  we assume that the space pair $ (\mathbf{U}_h, Z_h)$ satisfy the inf-sup condition, i.e.  there exists a positive constant such that for $\forall w_{q_h}\in Z_h$
\begin{equation}\label{inf-sup-discrete}
\begin{aligned}
\sup_{0\neq\mathbf{v}_h\in\mathbf{U}_h}\frac{b(\mathbf{v}_h, w_{q_h})}{\vert\mathbf{v_h}\vert_1}\geq C\Vert w_{q_h}\Vert.
\end{aligned}
\end{equation}

\subsection{ Error estimate for the semi-discrete scheme}\label{sec:3.1}
For the semi-discrete scheme, we denote the error in displacement by $e_h^{\mathbf{u}}=\mathbf{u}-\mathbf{u}_h=\eta_h^{\mathbf{u}}+\xi_h^{\mathbf{u}}$, where $\eta_h^{\mathbf{u}}=\mathbf{u}-Q_h^{\mathbf{u}} \mathbf{u}$ and
$\xi_h^{\mathbf{u}}=Q_h^{\mathbf{u}} \mathbf{u}-\mathbf{u}_h$. Similarly, we decompose the errors of the total stress and pressure into two parts, respectively,  i.e. $e_h^q=\eta_h^q+\xi_h^q$ and $e_h^p=\eta_h^p+\xi_h^p$. 
\begin{theorem} \label{FEM1_th1-semi}
Assume that the inf-sup condition \eqref{inf-sup-discrete} is satisfied for the space pair $(\mathbf{U}_h, Z_h)$. Let $(\mathbf{u},~q,~p)$ and $(\mathbf{u}_h,~q_h,~p_h)$ be the solutions of \eqref{weakform_FEM1} and \eqref{FEM1_semi} respectively. Then there exists a constant such that for each $t\in(0,\bar{T}]$
\begin{equation}\label{FEM1_semi_erreq3.0}
\begin{aligned}
 &\mu\Vert\varepsilon(e_h^{\mathbf{u}}(t))\Vert^2+\Vert e_h^q(t)\Vert^2\leq C\Big(\mu\Vert\varepsilon(\eta_h^{\mathbf{u}}(t))\Vert^2+\Vert \eta_h^q(t)\Vert^2
 +\int_0^t \Vert \eta^q_{ht}-\eta^p_{ht}\Vert^2 ds\Big),
\end{aligned}
\end{equation}
and 
\begin{equation}\label{FEM1_semi_erreq3.1}
\begin{aligned}
&\kappa\Vert\nabla e_h^p(t)\Vert^2
\leq C\Big(\kappa\Vert\nabla \eta_h^p(t)\Vert^2+\int_0^t \Vert \eta^q_{ht}-\eta^p_{ht}\Vert^2 ds\Big).
\end{aligned}
\end{equation}
\end{theorem}
\begin{proof}
Subtracting \eqref{FEM1_semi} from \eqref{weakform_FEM1}, and taking the derivative of the second equation with respect to time $t$, we get 
\begin{equation*}
\begin{aligned}
2\mu(\varepsilon(e_h^{\mathbf{u}}), \varepsilon(\mathbf{v}))-(e^q_h,\nabla\cdot\mathbf{v})&=0, ~~~~~\forall \mathbf{v}\in\mathbf{U}_{h},\\
(\lambda^{-1}(e^q_{ht}-e^p_{ht}),w_q)+(\nabla\cdot e^{\mathbf{u}}_{ht},w_q)&=0,~~~~\forall w_q\in Z_{h},\\
-(\lambda^{-1}(e^q_{ht}-e^p_{ht}),w_p)+\kappa(\nabla e^p_h,\nabla w_p)&=0,~~~~\forall w_p\in P_{h}.\\
\end{aligned}
\end{equation*}
With the use of the definitions of projections \eqref{stoke_proj} and \eqref{elliptic_proj}, it follows
\begin{equation}\label{FEM1_semi_erreq1}
\begin{aligned}
2\mu(\varepsilon(\xi_h^{\mathbf{u}}), \varepsilon(\mathbf{v}))-(\xi^q_h,\nabla\cdot\mathbf{v})&=0, \\
(\lambda^{-1}(\xi^q_{ht}-\xi^p_{ht}),w_q)+(\nabla\cdot \xi^{\mathbf{u}}_{ht},w_q)&=-(\lambda^{-1}(\eta^q_{ht}-\eta^p_{ht}),w_q),\\
-(\lambda^{-1}(\xi^q_{ht}-\xi^p_{ht}),w_p)+\kappa(\nabla \xi^p_h,\nabla w_p)&=(\lambda^{-1}(\eta^q_{ht}-\eta^p_{ht}),w_p).\\
\end{aligned}
\end{equation}
Taking $\mathbf{v}=\xi_{ht}^{\mathbf{u}}$, $w_q=\xi_h^q$ and $w_p=\xi_h^p$, then we can deduce that  
\begin{equation*}
\begin{aligned}
\mu\frac{d}{dt}\Vert\varepsilon(\xi_h^{\mathbf{u}})\Vert^2+\frac{\lambda^{-1}}{2}\frac{d}{dt}\Vert\xi_h^q-\xi_h^p\Vert^2+\kappa\Vert\nabla\xi_h^p\Vert^2
&=-\lambda^{-1}(\eta^q_{ht}-\eta^p_{ht}, \xi_h^q-\xi_h^p)\\
&\leq\frac{\lambda^{-1}}{2}\Vert\eta^q_{ht}-\eta^p_{ht}\Vert^2+\frac{\lambda^{-1}}{2}\Vert\xi^q_{h}-\xi^p_{h}\Vert^2.
\end{aligned}
\end{equation*}
Next, integrating over $(0, t)$, since $\xi_h^{\mathbf{u}}(0)=0$ and $\xi_h^q(0)=0, \xi_h^p(0)=0$, we have 
\begin{equation*}
\begin{aligned}
&\mu\Vert\varepsilon(\xi_h^{\mathbf{u}}(t))\Vert^2+\frac{\lambda^{-1}}{2}\Vert\xi_h^q(t)-\xi_h^p(t)\Vert^2 +\kappa\int_0^t\Vert\nabla\xi_h^p\Vert^2ds\\
\leq& \frac{\lambda^{-1}}{2}\int_0^t \Vert \eta^q_{ht}-\eta^p_{ht}\Vert^2 ds+\frac{\lambda^{-1}}{2}\int_0^t\Vert\xi^q_{h}-\xi^p_{h}\Vert^2ds.
\end{aligned}
\end{equation*}
Using the Gronwall Lemma \cite{Quarteroni2008}, one finds 
\begin{equation}\label{FEM1_semi_erreq2}
\begin{aligned}
\mu\Vert\varepsilon(\xi_h^{\mathbf{u}}(t))\Vert^2+\frac{\lambda^{-1}}{2}\Vert\xi_h^q(t)-\xi_h^p(t)\Vert^2+\kappa\int_0^t\Vert\nabla\xi_h^p\Vert^2ds
\leq C \int_0^t \Vert \eta^q_{ht}-\eta^p_{ht}\Vert^2 ds.
\end{aligned}
\end{equation}

Now, to estimate $\Vert \xi_h^q\Vert$, we deduce from the first equation of \eqref{FEM1_semi_erreq1}, for each $\mathbf{v}\in\mathbf{U}_{h}$
\begin{equation*}
\begin{aligned}
b(\mathbf{v}, \xi_h^q)=(\xi_h^q, \nabla\cdot\mathbf{v})=2\mu(\varepsilon(\xi_h^{\mathbf{u}}),\varepsilon({\mathbf{v}} )).
\end{aligned}
\end{equation*}
Thus, using the inf-sup condition \eqref{inf-sup-discrete}, it follows that
\begin{equation}\label{FEM1_semi_eq}
\begin{aligned}
\Vert \xi_h^q\Vert\leq C\sup_{\vert \mathbf{v}\vert_1\neq 0}\frac{(\xi_h^q, \nabla\cdot\mathbf{v})}{\vert \mathbf{v}\vert_1}
=C\sup_{\vert \mathbf{v}\vert_1\neq 0}\frac{2\mu(\varepsilon(\xi_h^{\mathbf{u}}),\varepsilon({\mathbf{v}} ))}{\vert \mathbf{v}\vert_1}\leq C\Vert \varepsilon(\xi_h^{\mathbf{u}})\Vert.
\end{aligned}
\end{equation}

On the other hand, differentiating the first equation of  \eqref{FEM1_semi_erreq1}  with respect to $t$, it follows 
\begin{equation*}
\begin{aligned}
2\mu(\varepsilon(\xi_{ht}^{\mathbf{u}}), \varepsilon(\mathbf{v}))-(\xi^q_{ht},\nabla\cdot\mathbf{v})&=0, ~~~~~\forall \mathbf{v}\in\mathbf{U}_{h}.
\end{aligned}
\end{equation*}
Taking $\mathbf{v}=\xi_{ht}^{\mathbf{u}}$ in the above identity and $w_q=\xi_{ht}^q$, $w_p=\xi_{ht}^p$ in \eqref{FEM1_semi_erreq1}, we have 
\begin{equation*}
\begin{aligned}
2\mu\Vert\varepsilon(\xi_{ht}^{\mathbf{u}})\Vert^2+\lambda^{-1}\Vert\xi_{ht}^q-\xi_{ht}^p\Vert^2+\frac{\kappa}{2}\frac{d}{dt}\Vert\nabla\xi_h^p\Vert^2
&=-\lambda^{-1}(\eta^q_{ht}-\eta^p_{ht}, \xi_{ht}^q-\xi_{ht}^p)\\
&\leq\frac{\lambda^{-1}}{2}\Vert\eta^q_{ht}-\eta^p_{ht}\Vert^2+\frac{\lambda^{-1}}{2}\Vert\xi^q_{ht}-\xi^p_{ht}\Vert^2.
\end{aligned}
\end{equation*}
So we rewrite the above inequality as following
\begin{equation*}
\begin{aligned}
2\mu\Vert\varepsilon(\xi_{ht}^{\mathbf{u}})\Vert^2+\frac{\lambda^{-1}}{2}\Vert\xi_{ht}^q-\xi_{ht}^p\Vert^2+\frac{\kappa}{2}\frac{d}{dt}\Vert\nabla\xi_h^p\Vert^2
\leq \frac{\lambda^{-1}}{2}\Vert\eta^q_{ht}-\eta^p_{ht}\Vert^2.
\end{aligned}
\end{equation*}
Integrating on $(0,t)$, we have
\begin{equation}\label{FEM1_semi_erreq3}
\begin{aligned}
4\mu\int_0^t\Vert\varepsilon(\xi_{ht}^{\mathbf{u}})\Vert^2ds+\lambda^{-1}\int_0^t\Vert\xi_{ht}^q-\xi_{ht}^p\Vert^2ds+\kappa\Vert\nabla\xi_h^p(t)\Vert^2
\leq C\int_0^t\Vert\eta^q_{ht}-\eta^p_{ht}\Vert^2ds.
\end{aligned}
\end{equation}
The proof is completed combining \eqref{FEM1_semi_erreq2}, \eqref{FEM1_semi_eq}, \eqref{FEM1_semi_erreq3} with the definitions of errors.
\end{proof}
Moreover, as an immediate application of \eqref{FEM1_semi_erreq2},  \eqref{FEM1_semi_eq}, and \eqref{FEM1_semi_erreq3}, we have the following corollary.
\begin{corollary} \label{FEM1_co1-semi}
Assume that the inf-sup condition \eqref{inf-sup-discrete} is satisfied for the space pair $(\mathbf{U}_h, Z_h)$. Let $(\mathbf{u},~q,~p)$ and $(\mathbf{u}_h,~q_h,~p_h)$ be the solutions of \eqref{weakform_FEM1} and \eqref{FEM1_semi}, then there exists a constant such that for each $t\in(0,\bar{T}]$
\begin{equation*}
\begin{aligned}
 &\int_0^t (\mu\Vert\varepsilon(e_{ht}^{\mathbf{u}})\Vert^2+\Vert e_{ht}^q\Vert^2)ds\leq C\int_0^t(\mu\Vert\varepsilon(\eta_{ht}^{\mathbf{u}})\Vert^2+\Vert \eta_{ht}^q\Vert^2
 + \Vert \eta^q_{ht}-\eta^p_{ht}\Vert^2 )ds,
\end{aligned}
\end{equation*}
and 
\begin{equation*}
\begin{aligned}
&\kappa\int_0^t \Vert\nabla e_h^p\Vert^2
\leq C\int_0^t (\kappa\Vert\nabla \eta_h^p\Vert^2+\int_0^t \Vert \eta^q_{ht}-\eta^p_{ht}\Vert^2 )ds.
\end{aligned}
\end{equation*}
\end{corollary}
\subsection{Error estimate for the fully-discrete scheme}\label{sec:3.2}
 Denote the error in displacement at time $t^n$ by $e_{\mathbf{u}}^n=\mathbf{u}(t^n)-\mathbf{u}^n=\mathbf{u}(t^n)-Q_h^{\mathbf{u}}\mathbf{u}(t^n)+Q_h^{\mathbf{u}}\mathbf{u}(t^n)-\mathbf{u}^n:=\eta_{\mathbf{u}}^n+\xi_{\mathbf{u}}^n$.
We define $e_q^n$ and $e_p^n$ in a similar fashion as $e_{\mathbf{u}}^n$.
\begin{theorem} \label{FEM1_th2-fully}
Assume that the inf-sup condition \eqref{inf-sup-discrete} is satisfied for the space pair $(\mathbf{U}_h, Z_h)$. Let $(\mathbf{u},~q,~p)$ and $(\mathbf{u}^n,~q^n,~p^n)$ be the solutions of \eqref{weakform_FEM1} and \eqref{FEM1_fully}, then there exists a constant such that for $t^N\in(0,\bar{T})$
\begin{equation}\label{FEM1_fully-err-eq1}
\begin{aligned}
\mu\Vert\varepsilon(e_{\mathbf{u}}^N)\Vert^2+\Vert e_q^N\Vert^2
\leq &C\Big(\mu\Vert\varepsilon(\eta_{\mathbf{u}}^N)\Vert^2+\Vert \eta_q^N\Vert^2+\tau \int_{0}^{t^N}(\Vert \eta_{ht}^q\Vert^2+\Vert  \eta_{ht}^p\Vert^2)ds\\
&+\tau^3\int_{0}^{t^N}(\Vert q_{tt}\Vert^2+\Vert p_{tt}\Vert^2)ds+\tau^2\int_{0}^{t^N}\Vert \mathbf{u}_{tt}\Vert_1^2ds\Big),\\
\end{aligned}
\end{equation}
and
\begin{equation}\label{FEM1_fully-err-eq2}
\begin{aligned}
\kappa\Vert\nabla e_p^N\Vert^2&\leq \kappa\Vert\nabla \eta_p^N\Vert^2+C\Big(\tau^2\int_{0}^{t^N}(\Vert q_{tt}\Vert^2+\Vert p_{tt}
\Vert^2+\Vert\mathbf{u}_{tt}\Vert_1^2)ds\\
&+\int_{0}^{t^N}(\Vert \eta_{ht}^q\Vert^2+\Vert  \eta_{ht}^p\Vert^2)ds\Big).
\end{aligned}
\end{equation}
\end{theorem}
\begin{proof}
For each $\mathbf{v}\in\mathbf{U}_h$, using the definition of Stokes projection \eqref{stoke_proj},  we have 
\begin{equation*}
\begin{aligned}
2\mu(\varepsilon(Q_h^{\mathbf{u}}\mathbf{u}(t^n)), \varepsilon(\mathbf{v}))-(Q_h^q q(t^n),\nabla\cdot\mathbf{v})=&2\mu(\varepsilon(\mathbf{u}(t^n)), \varepsilon(\mathbf{v}))-(q(t^n),\nabla\cdot\mathbf{v})\\
=&(\mathbf{f}^n,\mathbf{v})+\langle \beta^n, \mathbf{v}\rangle_{\Gamma_N}.
\end{aligned}
\end{equation*}
Combining the above identity with the first equation of \eqref{FEM1_fully} yields
\begin{equation}\label{FEM1_fully-discrete-eq1}
\begin{aligned}
2\mu(\varepsilon(\xi_{\mathbf{u}}^n), \varepsilon(\mathbf{v}))-(\xi_q^n,\nabla\cdot\mathbf{v})=0.
\end{aligned}
\end{equation}
For each $w_q\in Z_h$,  taking into account the derivative of the second equation in \eqref{weakform_FEM1} with respect to time $t$ and \eqref{stoke_proj},  one obtains
\begin{equation*}
\begin{aligned}
&(\lambda^{-1}(\bar{\partial}_t Q_h^q q(t^n)-\bar{\partial}_t Q_h^p p(t^n)),w_q)+(\nabla\cdot\bar{\partial}_t{Q_h^{\mathbf{u}}\mathbf{u}(t^n)},w_q)\\
=&-\lambda^{-1}(q_t(t^n)-\bar{\partial}_t q(t^n),w_q)+\lambda^{-1}(p_t(t^n)-\bar{\partial}_t p(t^n),w_q)
-\lambda^{-1}(\bar{\partial}_t (I-Q_h^q) q(t^n),w_q)\\
&+\lambda^{-1}(\bar{\partial}_t (I-Q_h^p) p(t^n),w_q)-(\nabla\cdot(\mathbf{u}_t-\bar{\partial}_t \mathbf{u})(t^n), w_q).
\end{aligned}
\end{equation*}
Rewriting the second term of \eqref{FEM1_fully} as $(\lambda^{-1}(\bar{\partial}_t q^n-\bar{\partial}_t p^n),w_q)+(\nabla\cdot\bar{\partial}_t{\mathbf{u}^n},w_q)=0$ and substituting it into the above identity lead to 
 \begin{equation}\label{FEM1_fully-discrete-eq2}
\begin{aligned}
&(\lambda^{-1}(\bar{\partial}_t \xi_q^n-\bar{\partial}_t  \xi_p^n),w_q)+(\nabla\cdot\bar{\partial}_t \xi_{\mathbf{u}}^n,w_q)\\
=&-\lambda^{-1}(q_t(t^n)-\bar{\partial}_t q(t^n),w_q)+\lambda^{-1}(p_t(t^n)-\bar{\partial}_t p(t^n),w_q)
-\lambda^{-1}(\bar{\partial}_t (I-Q_h^q) q(t^n),w_q)\\
&+\lambda^{-1}(\bar{\partial}_t (I-Q_h^p) p(t^n),w_q)-(\nabla\cdot(\mathbf{u}_t-\bar{\partial}_t \mathbf{u})(t^n), w_q).
\end{aligned}
\end{equation}
For each $w_p\in P_h$,  the definition of elliptic projection \eqref{elliptic_proj} and third equation of \eqref{weakform_FEM1} imply that
\begin{equation*}
\begin{aligned}
&-(\lambda^{-1}(\bar{\partial}_t Q_h^q q(t^n)-\bar{\partial}_t Q_h^p p(t^n)),w_p)+\kappa(\nabla Q_h^p  p(t^n),\nabla w_p)\\
=&-\lambda^{-1}(q_t(t^n)-p_t(t^n), w_p)+\kappa(\nabla p(t^n),\nabla w_p)+\lambda^{-1}(q_t(t^n)-\bar{\partial}_t q(t^n),w_p)\\
&-\lambda^{-1}(p_t(t^n)-\bar{\partial}_t p(t^n),w_p)+\lambda^{-1}(\bar{\partial}_t (I-Q_h^q) q(t^n),w_p)\\
&-\lambda^{-1}(\bar{\partial}_t (I-Q_h^p) p(t^n),w_p)\\
=&(g^n, w_p)+\langle \gamma^n, w_p\rangle_{\Gamma_N}+\lambda^{-1}(q_t(t^n)-\bar{\partial}_t q(t^n),w_p)-\lambda^{-1}(p_t(t^n)-\bar{\partial}_t p(t^n),w_p)\\
&+\lambda^{-1}(\bar{\partial}_t (I-Q_h^q) q(t^n),w_p)-\lambda^{-1}(\bar{\partial}_t (I-Q_h^p) p(t^n),w_p).
\end{aligned}
\end{equation*}
Therefore, by employing the third equation of \eqref{FEM1_fully}, one finds
 \begin{equation}\label{FEM1_fully-discrete-eq3}
\begin{aligned}
&-(\lambda^{-1}(\bar{\partial}_t \xi_q^n-\bar{\partial}_t \xi_p^n),w_p)+\kappa(\nabla  \xi_p^n,\nabla w_p)\\
=&\lambda^{-1}(q_t(t^n)-\bar{\partial}_t q(t^n),w_p)-\lambda^{-1}(p_t(t^n)-\bar{\partial}_t p(t^n),w_p)\\
&+\lambda^{-1}(\bar{\partial}_t (I-Q_h^q) q(t^n),w_p)-\lambda^{-1}(\bar{\partial}_t (I-Q_h^p) p(t^n),w_p).
\end{aligned}
\end{equation}
To proceed  our analysis, taking $\mathbf{v}=\tau\bar{\partial}_t \xi_{\mathbf{u}}^n$, $w_q=\tau\xi_q^n$ and $w_p=\tau\xi_p^n$ in \eqref{FEM1_fully-discrete-eq1}, \eqref{FEM1_fully-discrete-eq2} and \eqref{FEM1_fully-discrete-eq3} respectively, we have 
 \begin{equation*}
\begin{aligned}
LEH:=&2\mu(\varepsilon(\xi_{\mathbf{u}}^n), \varepsilon(\tau\bar{\partial}_t \xi_{\mathbf{u}}^n))+(\lambda^{-1}(\bar{\partial}_t \xi_q^n-\bar{\partial}_t  \xi_p^n),\tau\xi_q^n-\tau\xi_p^n)+\tau\kappa\Vert\nabla  \xi_p^n\Vert^2\\
=&-\lambda^{-1}(q_t(t^n)-\bar{\partial}_t q(t^n),\tau\xi_q^n-\tau\xi_p^n)
+\lambda^{-1}(p_t(t^n)-\bar{\partial}_t p(t^n),\tau\xi_q^n-\tau\xi_p^n)\\
&-\lambda^{-1}(\bar{\partial}_t (I-Q_h^q) q(t^n),\tau\xi_q^n-\tau\xi_p^n)
+\lambda^{-1}(\bar{\partial}_t (I-Q_h^p) p(t^n),\tau\xi_q^n-\tau\xi_p^n)\\
&-(\nabla\cdot(\mathbf{u}_t-\bar{\partial}_t )\mathbf{u}(t^n), \tau\xi_q^n)\\
=&:REH.
\end{aligned}
\end{equation*}
Note that, using Cauchy–Schwarz inequality and Poincar$\acute{e}$ inequality, it follows
 \begin{equation*}
\begin{aligned}
REH\leq& \tau^2\frac{\lambda^{-1}}{\epsilon_1}\Vert q_t(t^n)-\bar{\partial}_t q(t^n) \Vert^2+\tau^2\frac{\lambda^{-1}}{\epsilon_2}\Vert (p_t(t^n)-\bar{\partial}_t p(t^n)\Vert^2\\
&+\tau^2\frac{\lambda^{-1}}{\epsilon_3}\Vert  \bar{\partial}_t (I-Q_h^q) q(t^n)\Vert^2
+\tau^2\frac{\lambda^{-1}}{\epsilon_4}\Vert  \bar{\partial}_t (I-Q_h^p) p(t^n)\Vert^2\\
&+\lambda^{-1}\sum_{1\leq i\leq4}\frac{\epsilon_i}{4}\Vert  \xi_q^n-\xi_p^n \Vert^2
+\frac{\tau}{\epsilon_5}\Vert\nabla\cdot(\mathbf{u}_t- \bar{\partial}_t)\mathbf{u} (t^n)\Vert ^2+\frac{\epsilon_5}{4}\tau\Vert \xi_q^n\Vert ^2\\
\leq& C\tau^3\lambda^{-1}\int_{t^{n-1}}^{t^n}(\Vert q_{tt}\Vert^2+\Vert p_{tt}\Vert^2)ds+C\tau\lambda^{-1}\int_{t^{n-1}}^{t^n}(\Vert \eta_{ht}^q\Vert^2+\Vert  \eta_{ht}^p\Vert^2)ds\\
&+C\tau^2\int_{t^{n-1}}^{t^n}\Vert \mathbf{u}_{tt}\Vert_1^2ds
+(\lambda^{-1}\sum_{1\leq i\leq4}\frac{\epsilon_i}{4}+\frac{\epsilon_5\tau}{2})\Vert  \xi_q^n-\xi_p^n \Vert^2+C_p\frac{\epsilon_5}{2}\tau\Vert \nabla\xi_p^n\Vert ^2,
\end{aligned}
\end{equation*}
where $\epsilon_i,~i=1:5$ are positive constants. Here, we take into account the following inequalities,
 \begin{equation*}
\begin{aligned}
&\Vert q_t(t^n)-\bar{\partial}_t q(t^n) \Vert^2= \Vert\frac{1}{\tau}\int_{t^{n-1}}^{t^n} (s-t^{n-1})q_{tt} ds\Vert^2\leq \tau\int_{t^{n-1}}^{t^n} \Vert q_{tt}\Vert^2ds,\\
&\Vert  \bar{\partial}_t (I-Q_h^q) q(t^n)\Vert^2=\Vert\frac{1}{\tau}\int_{t^{n-1}}^{t^n} (I-Q_h^q)q_{t} ds\Vert^2\leq C\frac{1}{\tau}\int_{t^{n-1}}^{t^n} \Vert \eta_{ht}^q\Vert^2ds,\\
&\Vert \xi_q^n\Vert^2 \leq 2\Vert \xi_q^n- \xi_p^n\Vert^2+2\Vert \xi_p^n\Vert^2\leq 2\Vert \xi_q^n- \xi_p^n\Vert^2+2C_p\Vert \nabla\xi_p^n\Vert^2.
\end{aligned}
\end{equation*}
Choosing $\epsilon_5=\displaystyle\min\left\{\frac{2\kappa}{C_p}, \frac{\lambda^{-1}}{2\tau}\right\}$, $\epsilon_i=\frac{1}{4},~i=1:4$, then we can reformulate the above inequality as
 \begin{equation*}
\begin{aligned}
&\mu\Vert\varepsilon(\xi_{\mathbf{u}}^n)\Vert^2-\mu\Vert\varepsilon(\xi_{\mathbf{u}}^{n-1})\Vert^2+\frac{\lambda^{-1}}{2}\Vert\xi_q^n-\xi_p^n\Vert^2-\frac{\lambda^{-1}}{2}\Vert\xi_q^{n-1}-\xi_p^{n-1}\Vert^2+\tau\frac{\kappa}{2}\Vert\nabla  \xi_p^n\Vert^2\\
\leq& C\tau^3\lambda^{-1}\int_{t^{n-1}}^{t^n}(\Vert q_{tt}\Vert^2
+\Vert p_{tt}\Vert^2)ds+C\tau\lambda^{-1}\int_{t^{n-1}}^{t^n}(\Vert \eta_{ht}^q\Vert^2+\Vert \eta_{ht}^p\Vert^2)ds\\
&+C\tau^2\int_{t^{n-1}}^{t^n}\Vert \mathbf{u}_{tt}\Vert_1^2ds+\frac{\lambda^{-1}}{2}\Vert\xi_q^n-\xi_p^n\Vert^2,
\end{aligned}
\end{equation*}
where we use the estimate
 \begin{equation*}
\begin{aligned}
&\mu\Vert\varepsilon(\xi_{\mathbf{u}}^n)\Vert^2-\mu\Vert\varepsilon(\xi_{\mathbf{u}}^{n-1})\Vert^2+\frac{\lambda^{-1}}{2}\Vert\xi_q^n-\xi_p^n\Vert^2-\frac{\lambda^{-1}}{2}\Vert\xi_q^{n-1}-\xi_p^{n-1}\Vert^2+\tau\kappa\Vert\nabla  \xi_p^n\Vert^2\\
\leq& LEH.
\end{aligned}
\end{equation*}
Then, adding $n$ from $1$ to $N$, we get 
 \begin{equation*}
\begin{aligned}
&\mu\Vert\varepsilon(\xi_{\mathbf{u}}^N)\Vert^2+\frac{\lambda^{-1}}{2}\Vert\xi_q^N-\xi_p^N\Vert^2+\tau\frac{\kappa}{2}\sum_{n=1}^N\Vert\nabla  \xi_p^n\Vert^2\\
\leq& C\tau^3\int_{0}^{t^N}(\Vert q_{tt}\Vert^2
+\Vert p_{tt}\Vert^2)ds+C\tau \int_{0}^{t^N}(\Vert \eta_{ht}^q\Vert^2+\Vert  \eta_{ht}^p\Vert^2)ds\\
&+C\tau^2\int_{0}^{t^N}\Vert \mathbf{u}_{tt}\Vert_1^2ds+\frac{\lambda^{-1}}{2}\sum_{n=1}^N\Vert\xi_q^n-\xi_p^n\Vert^2.
\end{aligned}
\end{equation*}
Using discrete Gronwall Lemma \cite{Quarteroni2008},  we deduce 
\begin{equation}\label{FEM1_fully-discrete-eq4}
\begin{aligned}
&\mu\Vert\varepsilon(\xi_{\mathbf{u}}^N)\Vert^2+\frac{\lambda^{-1}}{2}\Vert\xi_q^N-\xi_p^N\Vert^2+\tau\frac{\kappa}{2}\sum_{n=1}^N\Vert\nabla  \xi_p^n\Vert^2\\
\leq& C\Big(\tau^3\int_{0}^{t^N}(\Vert q_{tt}\Vert^2
+\Vert p_{tt}\Vert^2)ds
+\tau \int_{0}^{t^N}(\Vert \eta_{ht}^q\Vert^2+\Vert  \eta_{ht}^p\Vert^2)ds
+\tau^2\int_{0}^{t^N}\Vert \mathbf{u}_{tt}\Vert_1^2ds\Big).
\end{aligned}
\end{equation}
Note that  $\Vert \xi_q^n\Vert\leq C2\mu\Vert\varepsilon(\xi_{\mathbf{u}}^n)\Vert$ follows from \eqref{FEM1_fully-discrete-eq2} and \eqref{inf-sup-discrete}, therefore,  it leads to  the error estimates \eqref{FEM1_fully-err-eq1}.

On the other side, rewriting  \eqref{FEM1_fully-discrete-eq1}  leads to for each $\mathbf{v}\in\mathbf{U}_h$
 \begin{equation}\label{FEM1_fully-discrete-eq1.1}
\begin{aligned}
2\mu(\varepsilon(\bar{\partial}_t\xi_{\mathbf{u}}^n), \varepsilon(\mathbf{v}))-(\bar{\partial}_t\xi_q^n,\nabla\cdot\mathbf{v})=0.
\end{aligned}
\end{equation}
Choosing $\mathbf{v}=\tau\bar{\partial}_t\xi_{\mathbf{u}}^n$, $w_q=\tau\bar{\partial}_t\xi_q^n$ and $w_q=\tau\bar{\partial}_t\xi_q^n$ 
in \eqref{FEM1_fully-discrete-eq1.1}, \eqref{FEM1_fully-discrete-eq2} and  \eqref{FEM1_fully-discrete-eq3}, thus we can get
 \begin{equation*}
\begin{aligned}
&2\mu\tau\Vert\varepsilon(\bar{\partial}_t\xi_{\mathbf{u}}^n) \Vert^2+\lambda^{-1}\tau\Vert\bar{\partial}_t \xi_q^n-\bar{\partial}_t  \xi_p^n\Vert^2
+\kappa\tau(\nabla  \xi_p^n, \nabla\bar{\partial}_t\xi_p^n)\\
=&-\Big(\nabla\cdot(\int_{t^{n-1}}^{t^n}\mathbf{u}_{tt}(t-t^{n-1})ds), \bar{\partial}_t\xi_q^n\Big)
-\lambda^{-1}\Big(\int_{t^{n-1}}^{t^n} (\eta^q_{ht}-\eta^p_{ht})ds,\bar{\partial}_t\xi_q^n-\bar{\partial}_t\xi_p^n\Big)\\
&-\lambda^{-1}\Big(\int_{t^{n-1}}^{t^n}(q_{tt}- p_{tt})(t-t^{n-1})ds,\bar{\partial}_t\xi_q^n-\bar{\partial}_t\xi_p^n\Big)=: REH_F.
\end{aligned}
\end{equation*}
with the use of Cauchy-Schwarz inequality,  we obtain
 \begin{equation*}
\begin{aligned}
& REH_F\leq C\tau^2\lambda^{-1}\int_{t^{n-1}}^{t^n}(\Vert q_{tt}\Vert^2+\Vert p_{tt}\Vert^2)ds+C\lambda^{-1}\int_{t^{n-1}}^{t^n}\Vert\eta^q_{ht}-\eta^p_{ht}\Vert^2ds\\
&+\frac{\epsilon_1+\epsilon_2}{4}\lambda^{-1}\tau\Vert\bar{\partial}_t\xi_q^n-\bar{\partial}_t\xi_p^n\Vert^2+C\tau^2\int_{t^{n-1}}^{t^n}\Vert\mathbf{u}_{tt}\Vert_1^2ds
+\epsilon_3\tau\Vert\bar{\partial}_t\xi_q^n\Vert^2.
\end{aligned}
\end{equation*}
 Using \eqref{FEM1_fully-discrete-eq1.1} and inf-sup condition, it leads to $\Vert\bar{\partial}_t\xi_q^n\Vert\leq C\Vert\varepsilon(\bar{\partial}_t\xi_{\mathbf{u}}^n)\Vert$. 
Then, we take $\epsilon_1=\epsilon_2=\frac{1}{2}$, and $\epsilon_3=\frac{\mu }{C}$ and find
\begin{equation*}
\begin{aligned}
&\mu\tau\Vert\varepsilon(\bar{\partial}_t\xi_{\mathbf{u}}^n) \Vert^2+\frac{\lambda^{-1}}{2}\tau\Vert\bar{\partial}_t \xi_q^n-\bar{\partial}_t  \xi_p^n\Vert^2
+\frac{\kappa}{2}\Vert\nabla  \xi_p^n\Vert^2 -\frac{\kappa}{2}\Vert\nabla  \xi_p^{n-1}\Vert^2\\
&\leq C\tau^2\int_{t^{n-1}}^{t^n}(\Vert q_{tt}\Vert^2+\Vert p_{tt}
\Vert^2)ds+C\int_{t^{n-1}}^{t^n}\Vert\eta^q_{ht}-\eta^p_{ht}\Vert^2ds
+C\tau^2\int_{t^{n-1}}^{t^n}\Vert\mathbf{u}_{tt}\Vert_1^2ds.
\end{aligned}
\end{equation*}
Adding $n$ from $1$ to $N$, we obtain
\begin{equation}\label{FEM1_fully-discrete-eq5}
\begin{aligned}
& \mu\tau\sum_{n=1}^N\Vert\varepsilon(\bar{\partial}_t\xi_{\mathbf{u}}^n) \Vert^2+\frac{\lambda^{-1}}{2}\tau\sum_{n=1}^N\Vert\bar{\partial}_t \xi_q^n-\bar{\partial}_t  \xi_p^n\Vert^2
+\frac{\kappa}{2}\Vert\nabla  \xi_p^N\Vert^2 \\
&\leq C\tau^2\int_{0}^{t^N}(\Vert q_{tt}\Vert^2+\Vert p_{tt}
\Vert^2+\Vert\mathbf{u}_{tt}\Vert_1^2)ds+C\int_{0}^{t^N}\Vert\eta^q_{ht}-\eta^p_{ht}\Vert^2ds.
\end{aligned}
\end{equation}
Thus, we complete the proof  \eqref{FEM1_fully-err-eq2} from \eqref{FEM1_fully-discrete-eq5}.
\end{proof}

\subsection{Applying to lower-order finite elements}\label{sec:3.3}
In this subsection, we discuss two lowest order finite element approximation by considering the finite element space for the total stress to be continuous and discontinuous, then we present the error estimates of them.
\subsubsection{Finite elements with discontinuous total stress for $k=2$, $l=0$}
Consider the spaces,
\begin{equation*}
\begin{aligned}
 \mathbf{U}_{h}:=&\{\mathbf{v}\in[H^1_{0,D}(\Omega)]^d\cap[C^0(\Omega)]^d:\mathbf{v}\mid_{T}\in[{P}_{2}(T)]^d, \forall T \in \mathcal{T}_{h}\},\\
 Z_{h}:=&\{q\in L^2(\Omega):q\mid_{T}\in{P}_{0}(T), \forall T \in \mathcal{T}_{h}\},\\
P_{h}:=&\{p\in H^1_{0,D}(\Omega)\cap C^0(\Omega):p\mid_{T}\in{P}_{1}(T), \forall T \in \mathcal{T}_{h}\}.\\
 \end{aligned}
\end{equation*}
The inf-sup condtition \eqref{inf-sup-discrete} for the space pair $(\mathbf{U}_h, Z_h)$ is known \cite{Brezzi1991}. With the use of the properties of Stokes projection (see Lemma 3.1 in \cite{murad1996}), it follows
\begin{equation*}
\begin{aligned}
\Vert(I-Q_h^{\mathbf{u}})\mathbf{u} \Vert_1^2+\Vert(I-Q_h^{q})q \Vert^2 \leq Ch^2(\Vert \mathbf{u} \Vert_2^2+\Vert q \Vert_1^2 ).
 \end{aligned}
\end{equation*}
Hence, combining the above inequality with properties of elliptic projection \cite{Quarteroni2008}, the error estimates of semi-discrete scheme in Theorem \ref{FEM1_th1-semi} with $([P_2]^d,P_0,P_1)$ elements can be rewritten as follows.
\begin{theorem} \label{FEM1_th1-semi_p1p0p1}
Let $(\mathbf{u},~q,~p)$ and $(\mathbf{u}_h,~q_h,~p_h)$ be the solutions of \eqref{weakform_FEM1} and \eqref{FEM1_semi}. Assume $\mathbf{u}\in L^{\infty}(0,\bar{T}; [H^2(\Omega)]^d)$,
 $q\in L^{\infty}(0,\bar{T}; H^1(\Omega))$,  $q_t\in L^2(0,\bar{T}; H^1(\Omega))$,  $p\in L^{\infty}(0,\bar{T}; H^2(\Omega))$,  $p_t\in L^2(0,\bar{T}; H^1(\Omega))$, then there exists a constant such that for each $t\in(0,\bar{T}]$
\begin{equation*}
\begin{aligned}
 &\mu\Vert\varepsilon(e_h^{\mathbf{u}}(t))\Vert^2+\Vert e_h^q(t)\Vert^2\leq Ch^2\Big(\Vert \mathbf{u}(t)\Vert_2^2+\Vert q(t)\Vert_1^2
 +\int_0^t (\Vert q_{t}\Vert_1^2+\Vert p_{t}\Vert_1^2 )ds\Big),
\end{aligned}
\end{equation*}
and 
\begin{equation*}
\begin{aligned}
&\kappa\Vert\nabla e_h^p(t)\Vert^2
\leq Ch^2\Big(\Vert p(t)\Vert_2^2+\int_0^t (\Vert q_{t}\Vert_1^2+\Vert p_{t}\Vert_1^2 ) ds\Big).
\end{aligned}
\end{equation*}
\end{theorem}
Similarly, we rewrite the error estimates of the fully-discrete scheme in Theorem \ref{FEM1_th2-fully} with $([P_2]^d,P_0,P_1)$ elements.
\begin{theorem} \label{FEM1_th2-fully_p2p0p1}
Let $(\mathbf{u},~q,~p)$ and $(\mathbf{u}^n,~q^n,~p^n)$ be the solutions of \eqref{weakform_FEM1} and \eqref{FEM1_fully}. Assume $\mathbf{u}\in L^{\infty}(0,\bar{T}; [H^2(\Omega)]^d)$, $\mathbf{u}_{tt}\in L^2(0,\bar{T}; [H^1(\Omega)]^d)$,
 $q\in L^{\infty}(0,\bar{T}; H^1(\Omega))$,  $q_t\in L^2(0,\bar{T}; H^1(\Omega))$, $q_{tt}\in L^2(0,\bar{T}; L^2(\Omega))$, $p\in L^{\infty}(0,\bar{T}; H^2(\Omega))$,  $p_t\in L^2(0,\bar{T}; H^1(\Omega))$, $p_{tt}\in L^2(0,\bar{T}; L^2(\Omega))$, then there exists a constant such that for $t^N\in(0,\bar{T}]$
\begin{equation*}
\begin{aligned}
\mu\Vert\varepsilon(e_{\mathbf{u}}^N)\Vert^2+\Vert e_q^N\Vert^2
\leq& C\Big(h^2\Vert \mathbf{u}(t^N)\Vert_2^2+h^2\Vert q(t^N)\Vert_1^2+\tau h^2 \int_{0}^{t^N}(\Vert q_{t}\Vert_1^2+\Vert  p_{t}\Vert_1^2)ds\\
&+\tau^3\int_{0}^{t^N}(\Vert q_{tt}\Vert^2+\Vert p_{tt}\Vert^2)ds+\tau^2\int_{0}^{t^N}\Vert \mathbf{u}_{tt}\Vert_1^2ds\Big),\\
\kappa\Vert\nabla e_p^N\Vert^2\leq& C\Big(h^2\Vert p(t^N)\Vert_2^2+h^2 \int_{0}^{t^N}(\Vert q_{t}\Vert_1^2+\Vert  p_{t}\Vert_1^2)ds\\
&+\tau^2\int_{0}^{t^N}(\Vert q_{tt}\Vert^2+\Vert p_{tt}\Vert^2+\Vert \mathbf{u}_{tt}\Vert_1^2)ds\Big).
\end{aligned}
\end{equation*}
\end{theorem}
\subsubsection{Lowest Taylor-Hood elements with the continuous total stress are chosen for $k=2, ~l=1$}
Consider the spaces,
\begin{equation*}
\begin{aligned}
 \mathbf{U}_{h}:=&\{\mathbf{v}\in[H^1_{0,D}(\Omega)]^d\cap[C^0(\Omega)]^d:\mathbf{v}\mid_{T}\in[{P}_{2}(T)]^d, \forall T \in \mathcal{T}_{h}\},\\
 Z_{h}:=&\{q\in L^2(\Omega) \cap C^0(\Omega): q\mid_{T}\in{P}_{1}(T), \forall T \in \mathcal{T}_{h}\},\\
P_{h}:=&\{p\in H^1_{0,D}(\Omega)\cap C^0(\Omega):p\mid_{T}\in{P}_{1}(T), \forall T \in \mathcal{T}_{h}\}.\\
 \end{aligned}
\end{equation*}
It is well known that the space pair $(\mathbf{U}_h, Z_h)$ satisfies the inf-sup condition \eqref{inf-sup-discrete}  (see \cite{Girault1979}). With the use of the properties of stoke projection (see Lemma 3.1 in \cite{murad1996}), it follows
\begin{equation*}
\begin{aligned}
\Vert(I-Q_h^{\mathbf{u}})\mathbf{u} \Vert_1^2+\Vert(I-Q_h^{q})q \Vert^2 \leq Ch^4(\Vert \mathbf{u} \Vert_3^2+\Vert q \Vert_2^2 ).
 \end{aligned}
\end{equation*}
The convergence orders for the fully-discrete scheme with $([P_2]^d,P_1,P_1)$ elements will be verified in numerical examples. Theorem \ref{FEM1_th2-fully} is rewritten as follows.

\begin{theorem} \label{FEM1_th2-fully_p2p1p1}
Let $(\mathbf{u},~q,~p)$ and $(\mathbf{u}^n,~q^n,~p^n)$ be the solutions of \eqref{weakform_FEM1} and \eqref{FEM1_fully}, Assume $\mathbf{u}\in L^{\infty}(0,\bar{T}; [H^3(\Omega)]^d)$, $\mathbf{u}_{tt}\in L^2(0,\bar{T}; [H^1(\Omega)]^d)$,
 $q\in L^{\infty}(0,\bar{T}; H^2(\Omega))$,  $q_t\in L^2(0,\bar{T}; H^1(\Omega))$, $q_{tt}\in L^2(0,\bar{T}; L^2(\Omega))$, $p\in L^{\infty}(0,\bar{T}; H^2(\Omega))$,  $p_t\in L^2(0,\bar{T}; H^1(\Omega))$, $p_{tt}\in L^2(0,\bar{T}; L^2(\Omega))$, then there exists a constant such that for $t^N\in(0,\bar{T}]$
\begin{equation*}
\begin{aligned}
\mu\Vert\varepsilon(e_{\mathbf{u}}^N)\Vert^2+\Vert e_q^N\Vert^2
\leq& C\Big(h^4\Vert \mathbf{u}(t^N)\Vert_3^2+h^4\Vert q(t^N)\Vert_2^2+\tau h^4 \int_{0}^{t^N}(\Vert q_{t}\Vert_1^2+\Vert  p_{t}\Vert_1^2)ds\\
&+\tau^3\int_{0}^{t^N}(\Vert q_{tt}\Vert^2+\Vert p_{tt}\Vert^2)ds+\tau^2\int_{0}^{t^N}\Vert \mathbf{u}_{tt}\Vert_1^2ds\Big),\\
\kappa\Vert\nabla e_p^N\Vert^2\leq& C\Big(h^2\Vert p(t^N)\Vert_2^2+ h^4 \int_{0}^{t^N}(\Vert q_{t}\Vert_1^2+\Vert  p_{t}\Vert_1^2)ds\\
&+\tau^2\int_{0}^{t^N}(\Vert q_{tt}\Vert^2+\Vert p_{tt}\Vert^2+\Vert \mathbf{u}_{tt}\Vert_1^2)ds\Big).
\end{aligned}
\end{equation*}
\end{theorem}
\begin{remark}\label{rem01}
From Theorem \ref{FEM1_th2-fully_p2p0p1} and Theorem \ref{FEM1_th2-fully_p2p1p1},  we know that the convergence orders for $\Vert\varepsilon(e_{\mathbf{u}}^N)\Vert$ and $\Vert e_q^N\Vert$ are $O(\tau+h)$ and  $O(\tau+h^2)$ with $([P_2]^d,P_0,P_1)$ elements and $([P_2]^d,P_1,P_1)$ elements, respectively. Moreover, the convergence $O(\tau+h)$ for pressure $\Vert \nabla e_p^N\Vert $ are obtained in two lowest finite elements.
\end{remark}
\begin{remark}\label{rem02}
The results for the lowest Taylor-Hood finite elements in Theorem \ref{FEM1_th2-fully_p2p1p1} can be extended to higher orders elements directly with the properties of the stokes projection and elliptic projection. In general, the convergence in Theorem \ref{FEM1_th2-fully_p2p0p1} and Theorem \ref{FEM1_th2-fully_p2p1p1} can be extended to many finite elements spaces pairs which satisfy the inf-sup conditions \eqref{inf-sup-discrete} with the properties of projections.
\end{remark}
\section{Numerical experiments}\label{sec:4}
The numerical examples are presented to confirm our convergence orders with respect to mesh size $h$ and time step $\tau$ for $([P_2]^d,P_0,P_1)$ elements and $([P_2]^d,P_1,P_1)$ elements for the benchmark problems. Let us define the linear Lagrange interpolation  for pressure by $I_h^p$. For simplicity, we still denote the error between linear Lagrange interpolation and numerical solutions by $e_p^n=I_h^p p(t^n)-p^n$ in the following tables.  The error estimates are divided into two parts, i.e. 
$$p(t^n)-p^n=p(t^n)-I_h^p p(t^n)+I_h^p p(t^n)-p^n. $$ 
Note that, we can obtain the convergence orders for $p(t^n)-p^n$ as $e_p^n$ based on the properties of Lagrange interpolation \cite{Quarteroni2008}. The errors for $\mathbf{u}$ and $q$ are defined in a similar way.
We denote the error in energy norm for the displacement $\mathbf{u}^n$ by $|||e_\mathbf{u}^n|||^2=||\varepsilon(e_\mathbf{u}^n)||^2$.
For the square domain, the uniform triangular mesh is employed, and Figure \ref{fig01} shows the computational grid with the meshes. 

\begin{figure}[!th]
\renewcommand{\captionfont}{\footnotesize}
\centering
\subfigure[~]{
\begin{minipage}{6.0cm}
\centering
\includegraphics[width = 4.5cm,height =4cm]{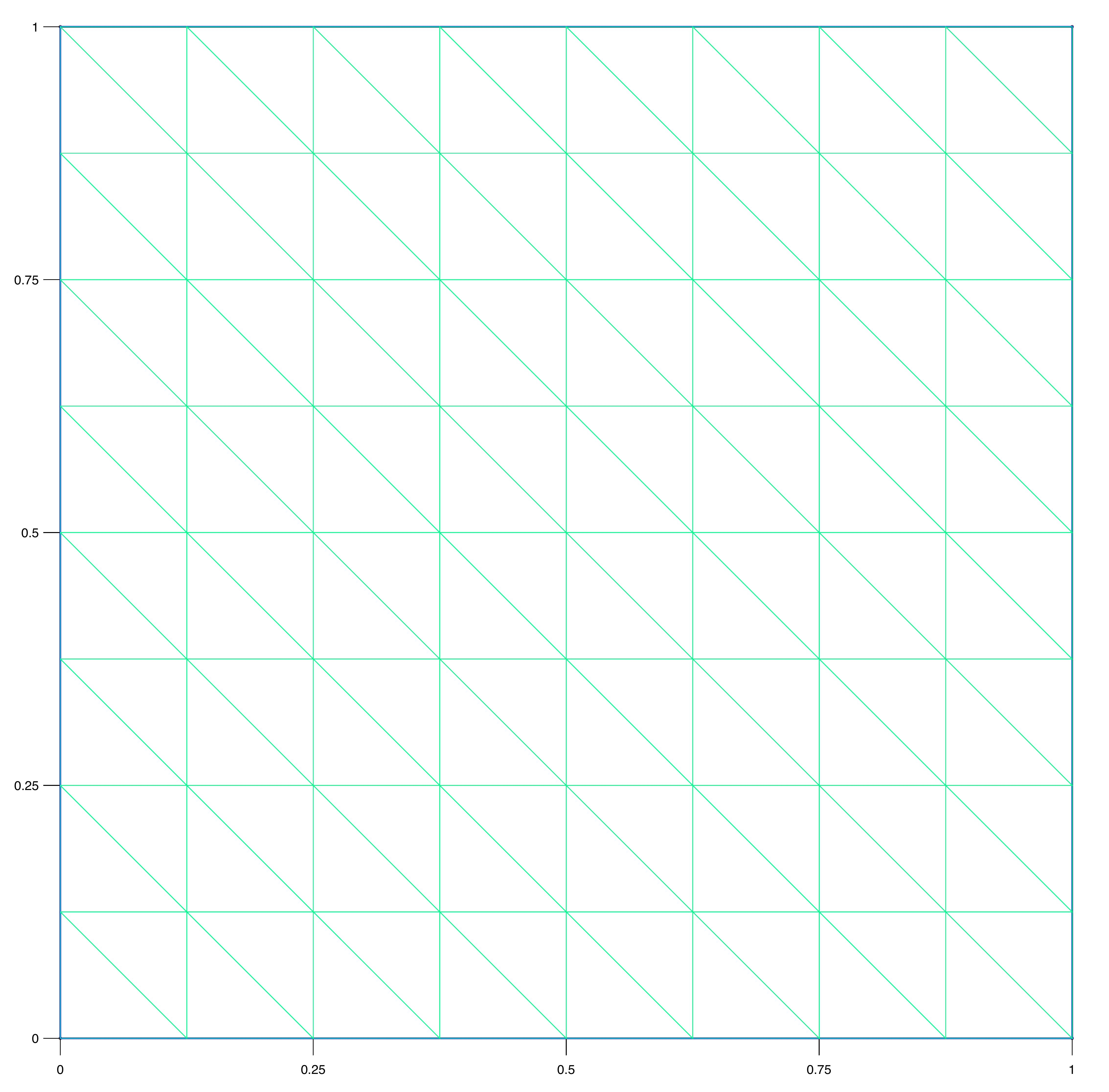}
\end{minipage}
}
\subfigure[~]{
\begin{minipage}{6cm}
\centering
\includegraphics[width = 4.5cm,height =4cm]{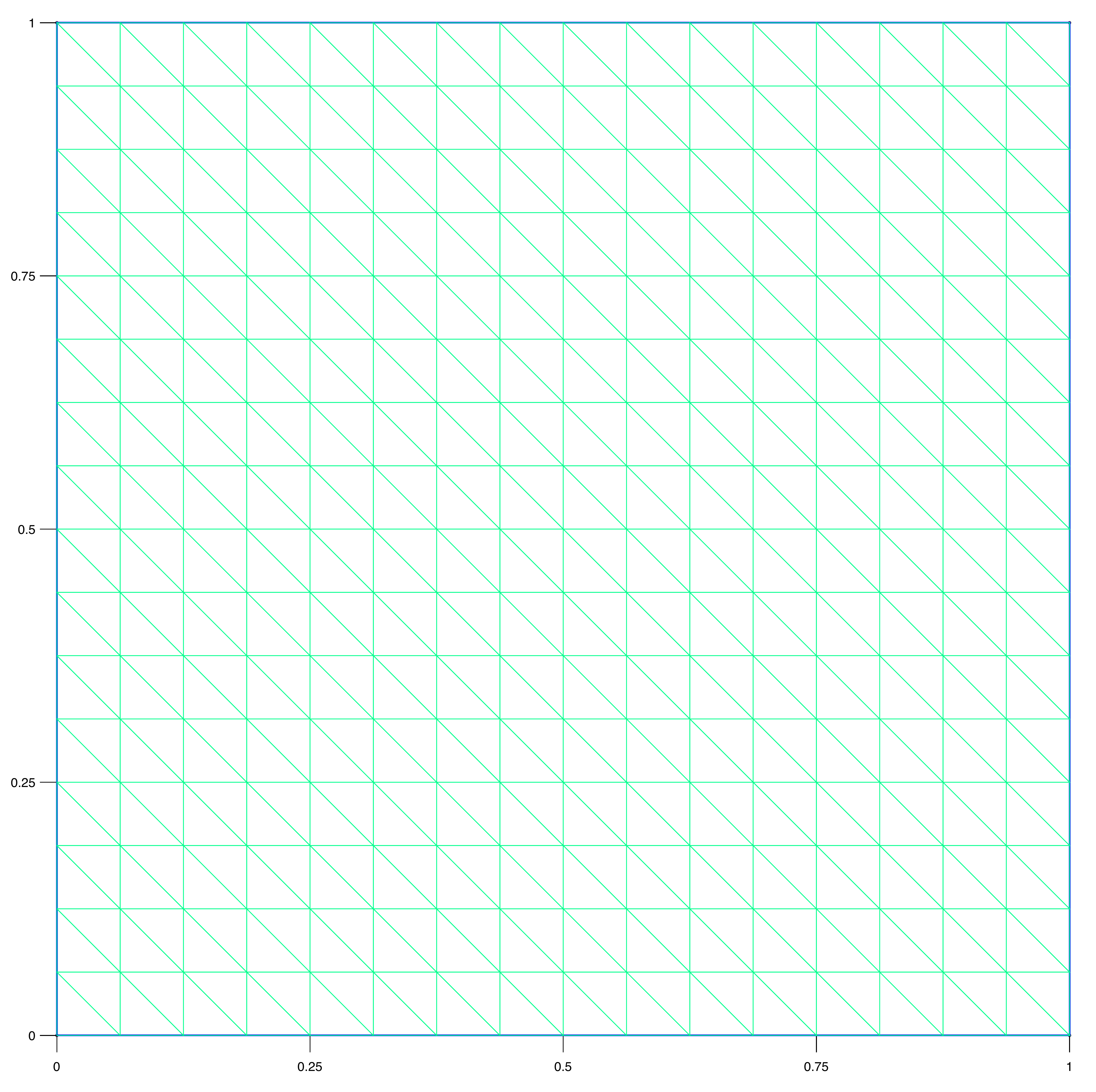}
\end{minipage}
}
\caption{The space meshes: (a) h=1/8, (b) h=1/16 }
\label{fig01}
\end{figure}

\begin{example}\label{ex:01}
Let the domain be $\Omega=(0,1)^2$ and $\bar{T}=1$. Then, we choose the Dirichlet boundary $\Gamma_D=\partial \Omega$ and the exact solutions 
\begin{equation*}
\mathbf{u}={t^2
\left[\begin{array}{c}
\sin(\pi x)\cos(\pi y)\\
\cos(\pi x)\sin(\pi y)
\end{array}
\right]
}
\end{equation*}
and 
\begin{equation*}
p=e^{-t}\sin(\pi x)\sin(\pi y).
\end{equation*}
The body force function $\mathbf{f}$, the forced fluid $g$, initial conditions and Dirichlet boundary conditions are determined by the exact solutions. We obtain 
$$\nabla\cdot\mathbf{u}=2 \pi t^2\cos(\pi x)\cos(\pi y).$$
\end{example}

We set the parameters $\kappa=1.0$, $\mu=1.0$ and $\lambda=10^{-2}$ in this example. Table \ref{TT01} gives the error estimates with respect to $h$ for $([P_2]^2,P_0,P_1)$ elements.
The convergence orders for each variable are consistent with our theory. Table \ref{TT02} shows the convergence orders with  $([P_2]^2,P_1,P_1)$ elements, which are  in agreements with 
our results in Theorem \ref{FEM1_th2-fully_p2p1p1}. Furthermore, the convergence orders with respect to $\tau$ can be observed in Table \ref{TT03} and Table \ref{TT04}.

 Moreover, the convergence orders for the pressure $\Vert\nabla e_p^n\Vert$ are computed to be $O(h^2)$ in Table \ref{TT01} and Table \ref{TT02},
 so we can obtain convergence results $O(h)$ for $\Vert \nabla(p(t^n)-p^n)\Vert$ due to the properties of the linear Lagrange interpolation.
A similar argument can be employed for the pressure in $H^1$ norm as well. In particular, the convergence orders of $O(h^2)$ for the displacement $\mathbf{u}$ and the pressure in $L^2$ norm can be observed in Tables \ref{TT01} and \ref{TT02}. 

\begin{table}[!th]
\renewcommand{\captionfont}{\footnotesize}
  \centering
   \resizebox{\textwidth}{!}{
  \small
  \begin{tabular}{cccccccccccc}
  \hline\noalign{\smallskip}
  \multicolumn{1}{c}{h}&\multicolumn{2}{c}{$|||e_\mathbf{u}^n|||$}
    &\multicolumn{2}{c}{$||e_\mathbf{u}^n||$}
      &\multicolumn{2}{c}{$||e_q^n||$}
      &\multicolumn{2}{c}{$||\nabla e_p^n||$}
      &\multicolumn{2}{c}{$||e_p^n||$}\\
\cline{2-3}\cline{4-5}\cline{6-7}\cline{8-9}\cline{10-11}\noalign{\smallskip}
 &error  &order &error  &order &error  &order &error  &order&error  &order\\
  \hline
  1/8     &1.2572e-02  &~~~~~ & 4.1887e-04   &~~~~~  &1.0502e-02&~~~~~  &7.8321e-02&~~~~~ &1.6727e-02  &~~~~~ \\
 1/16  &5.7283e-03  &1.1340 &9.7376e-05   &2.1048 &2.5910e-03&2.0190 &1.9241e-02 &2.0252&4.1523e-03  &2.0101   \\
 1/32  &2.8055e-03 &1.0298  &2.3932e-05   &2.0246 &6.4557e-04&2.0048 &4.7886e-03 &2.0065&1.0362e-03  &2.0026 \\
  1/64  &1.3961e-03 &1.0068 &5.9561e-06   &2.0065&1.6128e-04&2.0010 &1.1959e-03 &2.0015&2.5896e-04   &2.0005 \\
\hline
  \end{tabular}
  }
 \caption{Convergence at $t^n=1$ when $\tau=h^2$ for $([P_2]^2,P_0,P_1)$: Dirichlet boundary }
 \label{TT01}
\end{table}

\begin{table}[!th]
\renewcommand{\captionfont}{\footnotesize}
  \centering
   \resizebox{\textwidth}{!}{
  \small
  \begin{tabular}{cccccccccccc}
  \hline\noalign{\smallskip}
  \multicolumn{1}{c}{h}&\multicolumn{2}{c}{$|||e_\mathbf{u}^n|||$}
    &\multicolumn{2}{c}{$||e_\mathbf{u}^n||$}
      &\multicolumn{2}{c}{$||e_q^n||$}
      &\multicolumn{2}{c}{$||\nabla e_p^n||$}
      &\multicolumn{2}{c}{$||e_p^n||$}\\
\cline{2-3}\cline{4-5}\cline{6-7}\cline{8-9}\cline{10-11}\noalign{\smallskip}
 &error  &order &error  &order &error  &order &error  &order&error  &order\\
  \hline
 1/8    &3.8777e-03  &~~~~~ &3.0217e-04   &~~~~~ &2.8315e-03&~~~~~ &1.0661e-02 &~~~~~&2.3541e-03 &~~~~~   \\
 1/16  &6.8421e-04 &2.5026  &7.7262e-05   &1.9675 &7.2470e-04&1.9661 &2.6829e-03 &1.9904&6.0092e-04  &1.9699 \\
 1/32  &1.4486e-04 &2.2397 &1.9575e-05   &1.9807&1.8225e-04&1.9914 &6.7188e-04 &1.9975&1.5103e-04   &1.9923 \\
 1/64  &3.4293e-05  &2.0786 & 4.9123e-06 &1.9945 &4.5631e-05&1.9978 &1.6804e-04 &1.9993 &3.7807e-05 &1.9981 \\
\hline
  \end{tabular}
  }
 \caption{Convergence at $t^n=1$ when $\tau=h^2$ for $([P_2]^2P_1,P_1)$: Dirichlet boundary}
 \label{TT02}
\end{table}

\begin{table}[!th]
\renewcommand{\captionfont}{\footnotesize}
  \centering
  \resizebox{\textwidth}{!}{
  \small
  \begin{tabular}{cccccccccccc}
  \hline\noalign{\smallskip}
  \multicolumn{1}{c}{$\tau$}&\multicolumn{2}{c}{$|||e_\mathbf{u}^n|||$}
    &\multicolumn{2}{c}{$||e_\mathbf{u}^n||$}
      &\multicolumn{2}{c}{$||e_q^n||$}
      &\multicolumn{2}{c}{$||\nabla e_p^n||$}
      &\multicolumn{2}{c}{$||e_p^n||$}\\
\cline{2-3}\cline{4-5}\cline{6-7}\cline{8-9}\cline{10-11}\noalign{\smallskip}
 &error  &order &error  &order &error&order &error  &order \\
  \hline
 1   &3.8459e-03   &~~~~~ &2.4103e-03&~~~~~  &1.5573e-01&~~~~~ &3.3603e-01 &~~~~~&1.5887e-01 &~~~~~  \\
 1/2  &1.9797e-03 &0.9580 &1.2156e-03&0.9875 &7.8538e-02 &0.9875&1.6971e-01 &0.9855&8.0238e-02 &0.9854  \\
 1/4  &1.0634e-03 &0.8965 &6.0752e-04&1.0006 &3.9277e-02 &0.9997&8.5072e-02 &0.9963 &4.0241e-02 &0.9956 \\
  1/8  &6.5722e-04 &0.6942 &3.0332e-04&1.0020 &1.9639e-02 &0.9999&4.2737e-02 &0.9931&2.0245e-02 &0.9911 \\
\hline
  \end{tabular}
  }
 \caption{Convergence at $t^n=1$ when $h=1/64$ for $([P_2]^2,P_0,P_1)$: Dirichlet boundary}
 \label{TT03}
\end{table}

\begin{table}[!th]
\renewcommand{\captionfont}{\footnotesize}
  \centering
  \resizebox{\textwidth}{!}{
  \small
  \begin{tabular}{cccccccccccc}
  \hline\noalign{\smallskip}
  \multicolumn{1}{c}{$\tau$}&\multicolumn{2}{c}{$|||e_\mathbf{u}^n|||$}
    &\multicolumn{2}{c}{$||e_\mathbf{u}^n||$}
      &\multicolumn{2}{c}{$||e_q^n||$}
      &\multicolumn{2}{c}{$||\nabla e_p^n||$}
      &\multicolumn{2}{c}{$||e_p^n||$}\\
\cline{2-3}\cline{4-5}\cline{6-7}\cline{8-9}\cline{10-11}\noalign{\smallskip}
 &error  &order &error  &order &error&order &error  &order \\
  \hline
 1   &1.2052e-02   &~~~~~ &1.7118e-03&~~~~~  &2.9200e-02&~~~~~ &2.7565e-01 &~~~~~&2.9310e-02 &~~~~~  \\
 1/2  &6.0467e-03 &0.9950 &8.5831e-04&0.9959&1.4653e-02 &0.9947&1.3831e-01 &0.9949&1.4713e-02 &0.9943  \\
 1/4  &3.0175e-03 &1.0027 &4.2802e-04&1.0038 &7.3185e-03 &1.0015&6.9128e-02 &1.0005 &7.3538e-03 &1.0005 \\
  1/8  &1.5027e-03 &1.0057 &2.1286e-04&1.0077 &3.6508e-03 &1.0033&3.4532e-02 &1.0013&3.6736e-03 &1.0012\\
\hline
  \end{tabular}
  }
 \caption{Convergence at $t^n=1$ when $h=1/64$ for $([P_2]^2,P_1,P_1)$: Dirichlet boundary}
 \label{TT04}
\end{table}

\begin{example}\label{ex:01.1}
We consider the problem as in Example \ref{ex:01}  with Neumann and Dirichlet boundary conditions. Let the Neumann boundary be $\Gamma_N=\{(x, y)\mid x=1,  0<y<1\}$ and
the Dirichlet boundary $\Gamma_D=\partial \Omega \setminus \Gamma_N$. 
\end{example}

We also set the parameters $\kappa=1.0$, $\mu=1.0$ and $\lambda=10^{-2}$ as in Example \ref{ex:01}. With the Neumann and Dirichlet boundary conditions, we observe the convergence orders in energy norm of the displacement in Table \ref{TT05} and  Table \ref{TT06}  which are in full agreements with our main results in Theorem \ref{FEM1_th2-fully_p2p0p1} and Theorem  \ref{FEM1_th2-fully_p2p1p1}. The convergence orders $O(h^2)$ for the total stress in $L^2$ norm are recorded in Table \ref{TT05} and  Table \ref{TT06} which confirm our theory. 

\begin{table}[!th]
\renewcommand{\captionfont}{\footnotesize}
  \centering
   \resizebox{\textwidth}{!}{
  \small
  \begin{tabular}{cccccccccccc}
  \hline\noalign{\smallskip}
  \multicolumn{1}{c}{h}&\multicolumn{2}{c}{$|||e_\mathbf{u}^n|||$}
    &\multicolumn{2}{c}{$||e_\mathbf{u}^n||$}
      &\multicolumn{2}{c}{$||e_q^n||$}
      &\multicolumn{2}{c}{$||\nabla e_p^n||$}
      &\multicolumn{2}{c}{$||e_p^n||$}\\
\cline{2-3}\cline{4-5}\cline{6-7}\cline{8-9}\cline{10-11}\noalign{\smallskip}
 &error  &order &error  &order &error  &order &error  &order&error  &order\\
  \hline
 1/8    &1.4182e-02&~~~~  &1.7411e-03&~~~~   &1.2276e-02&~~~~   &8.0666e-02 &~~~~  &1.8625e-02 &~~~~   \\
 1/16  &5.9152e-03&1.2615 &4.1226e-04&2.0783 &3.0872e-03&1.9914 &2.0045e-02 &2.0087 &4.6735e-03&1.9946\\
 1/32  &2.8262e-03&1.0655 &9.9754e-05&2.0471  &7.6655e-04&2.0098 &4.9844e-03 &2.0077&1.1632e-03&2.0064  \\
 1/64  &1.3985e-03&1.0149  &2.4612e-05&2.0190  &1.9134e-04&2.0022 &1.2446e-03 &2.0017&2.9050e-04&2.0014\\
\hline
  \end{tabular}
  }
 \caption{Convergence at $t^n=1$ when $\tau=h^2$ for  $([P_2]^2,P_0,P_1)$: Neumann and Dirichlet boundary}
 \label{TT05}
\end{table}

\begin{table}[!th]
\renewcommand{\captionfont}{\footnotesize}
  \centering
   \resizebox{\textwidth}{!}{
  \small
  \begin{tabular}{cccccccccccc}
  \hline\noalign{\smallskip}
  \multicolumn{1}{c}{h}&\multicolumn{2}{c}{$|||e_\mathbf{u}^n|||$}
    &\multicolumn{2}{c}{$||e_\mathbf{u}^n||$}
      &\multicolumn{2}{c}{$||e_q^n||$}
      &\multicolumn{2}{c}{$||\nabla e_p^n||$}
      &\multicolumn{2}{c}{$||e_p^n||$}\\
\cline{2-3}\cline{4-5}\cline{6-7}\cline{8-9}\cline{10-11}\noalign{\smallskip}
 &error  &order &error  &order &error  &order &error  &order&error  &order\\
  \hline
1/8     &6.0055e-03&~~~~   &6.2900e-04&~~~~   &2.5456e-03 &~~~~   &1.5738e-02 &~~~~  &2.2821e-03  &~~~~     \\
1/16   &1.1116e-03&2.4336&1.6328e-04&1.9457 &6.5647e-04&1.9914 &3.9863e-03  &1.9811&5.8439e-04&1.9653 \\
1/32  &2.2324e-04&2.3159 &4.1244e-05&1.9850 &1.6542e-04&2.0098 &9.9993e-04 &1.9951&1.4701e-04 &1.9910  \\
1/64  &4.8554e-05&2.2009  &1.0324e-05&1.9981 &4.1437e-05&2.0022 &2.5020e-04 &1.9987&3.6809e-05&1.9977\\
\hline
  \end{tabular}
  }
 \caption{Convergence at $t^n=1$ when $\tau=h^2$ for  $([P_2]^2,P_1,P_1)$: Neumann and Dirichlet boundary}
 \label{TT06}
\end{table}

\begin{example}\label{ex:02}
Let the domain be $\Omega=(0,1)^2$ and $\bar{T}=1$. The Dirichlet boundary satisfy $\Gamma_D=\partial \Omega$.  Then, we choose the exact solutions 
\begin{equation*}
\mathbf{u}=e^{-t}{
\left[\begin{array}{c}
\sin(2\pi y)(-1+\cos(2\pi x))+\frac{1}{\mu+\lambda}\sin(\pi x)\sin(\pi y)\\
\sin(2\pi x)(1-\cos(2\pi y))+\frac{1}{\mu+\lambda}\sin(\pi x)\sin(\pi y)
\end{array}
\right]
}
\end{equation*}
and 
\begin{equation*}
p=e^{-t}\sin(\pi x)\sin(\pi y).
\end{equation*}
 We choose the body force function $\mathbf{f}$, forced fluid $g$, initial conditions and Dirichlet  boundary conditions so that  the exact solutions satisfy the problems \eqref{eq:problem1.1} and  \eqref{eq:problem1.2} .
Note that $\nabla\cdot\mathbf{u}= \pi e^{-t}\sin(\pi (x+y))/(\mu+\lambda)$, and when the elastic Lam$\acute{e}$ parameters $\lambda\rightarrow \infty$, we have $\nabla\cdot\mathbf{u}\rightarrow 0$.
\end{example}

We set the parameters $\kappa=1.0$, $\mu=1.0$ and $\lambda=10^4$ with nearly incompressible case. We test both $([P_2]^2,P_0,P_1)$ and $([P_2]^2,P_1,P_1)$ elements
 to emphasize the efficiency.  Thus, Table \ref{TT07} and Table \ref{TT08} show the convergence orders for each variable as expected.  At the same time, the convergence orders $O(h^2)$ for the total stress $\Vert e_q^n\Vert$ are shown in  Table \ref{TT07},  
 and we can observe that the convergence orders of displacement $|||e_\mathbf{u}^n|||$ are $O(h^3)$ in Table \ref{TT08} which is better than expected. Therefore, we can conclude that the finite element method is valid for the nearly incompressible case.
 
\begin{table}[!th]
\renewcommand{\captionfont}{\footnotesize}
  \centering
  \resizebox{\textwidth}{!}{
  \small
  \begin{tabular}{cccccccccccc}
  \hline\noalign{\smallskip}
  \multicolumn{1}{c}{h}&\multicolumn{2}{c}{$|||e_\mathbf{u}^n|||$}
    &\multicolumn{2}{c}{$||e_\mathbf{u}^n||$}
      &\multicolumn{2}{c}{$||e_q^n||$}
      &\multicolumn{2}{c}{$||\nabla e_p^n||$}
      &\multicolumn{2}{c}{$||e_p^n||$}\\
\cline{2-3}\cline{4-5}\cline{6-7}\cline{8-9}\cline{10-11}\noalign{\smallskip}
 &error  &order &error  &order &error  &order&error  &order&error  &order \\
  \hline
1/8     &4.8359e-02&~~~~   &1.1460e-03&~~~~ &1.0101e-02&~~~~   &1.0998e-02 &~~~~   &2.3094e-03&~~~~  \\
1/16    &1.9701e-02&1.2955&2.0102e-04&2.5111  &2.6388e-03&1.9365 &2.8017e-03 &1.9728&5.9488e-04&1.9568  \\
1/32  &9.7854e-03&1.0095 &4.9100e-05&2.0335&7.3649e-04&1.8411 &7.0387e-04 &1.9929  &1.4987e-04&1.9888 \\
1/64  &4.9240e-03&0.9908 &1.2386e-05&1.9870 &2.1109e-04&1.8028 &1.7619e-04  &1.9981&3.7540e-05&1.9972\\
\hline
  \end{tabular}
  }
 \caption{Convergence at $t^n=1$ when $\tau=h$ for  $([P_2]^2,P_0,P_1)$: Nearly incompressible case}
 \label{TT07}
\end{table}

\begin{table}[!th]
\renewcommand{\captionfont}{\footnotesize}
  \centering
   \resizebox{\textwidth}{!}{
  \small
  \begin{tabular}{cccccccccccc}
  \hline\noalign{\smallskip}
  \multicolumn{1}{c}{h}&\multicolumn{2}{c}{$|||e_\mathbf{u}^n|||$}
    &\multicolumn{2}{c}{$||e_\mathbf{u}^n||$}
      &\multicolumn{2}{c}{$||e_q^n||$}
      &\multicolumn{2}{c}{$||\nabla e_p^n||$}
      &\multicolumn{2}{c}{$||e_p^n||$}\\
\cline{2-3}\cline{4-5}\cline{6-7}\cline{8-9}\cline{10-11}\noalign{\smallskip}
 &error  &order &error  &order &error  &order&error  &order&error  &order \\
  \hline
1/8     &3.0635e-02&~~~~   &9.8803e-04&~~~~ &1.7083e-02&~~~~   &1.0998e-02 &~~~~   &2.3094e-03&~~~~  \\
1/16    &4.3311e-03&2.8223&6.9922e-05&3.8207  &3.0093e-03&2.5050 &2.8017e-03 &1.9728&5.9488e-04&1.9568 \\
1/32  &5.6515e-04&2.9380 &4.5564e-06&3.9397&7.0999e-04&2.0835 &7.0388e-04 &1.9929 &1.4987e-04&1.9888 \\
1/64  &7.1810e-05&2.9763 &2.8910e-07&3.9782 &1.7635e-04&2.0093 &1.7619e-04  &1.9981&3.7540e-05&1.9972\\
\hline
  \end{tabular}
  }
 \caption{Convergence at $t^n=1$ when $\tau=h$ for  $([P_2]^2,P_1,P_1)$: Nearly incompressible case}
 \label{TT08}
\end{table}

\begin{example}\label{ex:02.1}
We consider Example \ref{ex:02} with non-homogeneous Dirichlet boundary conditions. Let the domain be $\Omega=(0,1)^2$ and $\bar{T}=1$. The Dirichlet boundary satisfy $\Gamma_D=\partial \Omega$.  Then, we choose the exact solutions 
\begin{equation*}
\mathbf{u}=e^{-t}{
\left[\begin{array}{c}
\sin(2 y)(-1+\cos(2 x))+\frac{1}{\mu+\lambda}\sin( x)\sin( y)\\
\sin(2 x)(1-\cos(2 y))+\frac{1}{\mu+\lambda}\sin( x)\sin( y)
\end{array}
\right]
}
\end{equation*}
and 
\begin{equation*}
p=e^{-t}\sin( x)\sin(y).
\end{equation*}
Again, we choose the body force function $\mathbf{f}$, forced fluid $g$, initial conditions and Dirichlet boundary conditions so that  the exact solutions satisfy the problems \eqref{eq:problem1.1} and  \eqref{eq:problem1.2} .
Note that $\nabla\cdot\mathbf{u}= e^{-t}\sin( (x+y))/(\mu+\lambda)$, and when the elastic Lam$\acute{e}$ parameters $\lambda\rightarrow \infty$, we have $\nabla\cdot\mathbf{u}\rightarrow 0$.
\end{example}

We also set the parameters $\kappa=1.0$, $\mu=1.0$ and $\lambda=10^4$. The results in Table \ref{TT09} and Table \ref{TT010} verify our theory for the case with non-homogeneous Dirichlet boundary conditions.

\begin{table}[!th]
\renewcommand{\captionfont}{\footnotesize}
  \centering
   \resizebox{\textwidth}{!}{
  \small
  \begin{tabular}{cccccccccccc}
  \hline\noalign{\smallskip}
  \multicolumn{1}{c}{h}&\multicolumn{2}{c}{$|||e_\mathbf{u}^n|||$}
    &\multicolumn{2}{c}{$||e_\mathbf{u}^n||$}
      &\multicolumn{2}{c}{$||e_q^n||$}
      &\multicolumn{2}{c}{$||\nabla e_p^n||$}
      &\multicolumn{2}{c}{$||e_p^n||$}\\
\cline{2-3}\cline{4-5}\cline{6-7}\cline{8-9}\cline{10-11}\noalign{\smallskip}
 &error  &order &error  &order &error  &order&error  &order&error  &order \\
  \hline
1/8     &2.5465e-03&~~~~     &5.9338e-05&~~~~    &1.1059e-03&~~~~.    &1.6621e-04   &~~~~  &3.6140e-05&~~~~   \\
1/16   &1.3791e-03&0.8847  &1.7079e-05&1.7967 &3.4468e-04&1.6818 &4.2208e-05   &1.9774 &9.2862e-06 &1.9604\\
1/32   &7.1729e-04&0.9430  &4.6074e-06&1.8901 &1.0408e-04&1.7275 &1.0556e-05   &1.9994 &2.3293e-06&1.9951\\
1/64   &3.6553e-04&0.9725  &1.1977e-06&1.9436 &3.0911e-05&1.7515  &2.6206e-06  &2.0100  &5.7865e-07&2.0091  \\
\hline
  \end{tabular}
  }
 \caption{Convergence at $t^n=1$ when $\tau=h$ for  $([P_2]^2,P_0,P_1)$: Nearly incompressible case with non-homogeneous Dirichlet boundary conditions}
 \label{TT09}
\end{table}

\begin{table}[!th]
\renewcommand{\captionfont}{\footnotesize}
  \centering
   \resizebox{\textwidth}{!}{
  \small
  \begin{tabular}{cccccccccccc}
  \hline\noalign{\smallskip}
  \multicolumn{1}{c}{h}&\multicolumn{2}{c}{$|||e_\mathbf{u}^n|||$}
    &\multicolumn{2}{c}{$||e_\mathbf{u}^n||$}
      &\multicolumn{2}{c}{$||e_q^n||$}
      &\multicolumn{2}{c}{$||\nabla e_p^n||$}
      &\multicolumn{2}{c}{$||e_p^n||$}\\
\cline{2-3}\cline{4-5}\cline{6-7}\cline{8-9}\cline{10-11}\noalign{\smallskip}
 &error  &order &error  &order &error  &order&error  &order&error  &order \\
  \hline
1/8     &2.8381e-04&~~~~  &9.5582e-06&~~~~    &4.3295e-04&~~~~   &1.6621e-04   &~~~~&3.6141e-05&~~~~   \\
1/16   &3.8805e-05&2.8706 &6.4511e-07&3.8891 &1.0424e-04&2.0542 &4.2209e-05   &1.9773 &9.2864e-06&1.9604 \\
1/32   &5.0463e-06&2.9429 &4.1637e-08 &3.9536 &2.5944e-05&2.0064 &1.0556e-05   &1.9994 &2.3293e-06&1.9952\\
1/64   &6.4255e-07&2.9733  &2.6417e-09&3.9783 &6.4796e-06&2.0014 &2.6206e-06   &2.0100 &5.7866e-07&2.0091   \\
\hline
  \end{tabular}
  }
 \caption{Convergence at $t^n=1$ when $\tau=h$ for  $([P_2]^2,P_1,P_1)$: Nearly incompressible case with non-homogeneous Dirichlet boundary conditions}
 \label{TT010}
\end{table}

\section{Conclusions}\label{sec:5}
For Biot's consolidation model, we analyze the error estimates for a three field discretization.
And the total stress coupling between the solid and fluid variable is as an unknown. 
Moreover, the convergence of fully-discrete scheme is presented and the results can be extended to many finite element spaces that satisfy the inf-sup conditions.

\end{document}